\documentclass[10pt,twoside,a4paper]{article}
\usepackage{amsfonts,amssymb,amsmath,amscd,latexsym,makeidx,theorem}
\author{Luca Martinazzi\thanks{This work was supported by ETH Research Grant no. ETH-02 08-2.}
\\ \small ETH Zurich \\ \small R\"amistrasse 101, CH-8092 \\ \small \texttt{luca@math.ethz.ch}
}
\title{Concentration-compactness phenomena in the higher order Liouville's equation}
\date{September 12, 2008}

\newtheorem{trm}{Theorem}
\newtheorem{prop}[trm]{Proposition}
\newtheorem{cor}[trm]{Corollary}
\newtheorem{lemma}[trm]{Lemma}

\theorembodyfont{\upshape}

\theorembodyfont{\upshape}

\theorembodyfont{\upshape}

\newcommand{\R}[1]{\mathbb{R}^{#1}}

\newcommand{\de}{\partial}

\newcommand{\bs}{\backslash}
\newcommand{\dvol}{\,\mathrm{dvol}}

\newcommand{\M}[1]{\mathcal{#1}}
\newcommand{\ring}{r_{\mathrm{inj}}}
\newenvironment{proof}{\noindent\emph{Proof.}}{\hfill$\square$\medskip}

\DeclareMathOperator{\diver}{div}
\DeclareMathOperator{\loc}{loc}

\DeclareMathOperator*{\dist}{dist}

\DeclareMathOperator*{\Intm}{\int\!\!\!\!\!\! \rule[2.6pt]{6.5pt}{.4pt}}
\DeclareMathOperator{\intm}{\int\!\!\!\!\!\!\: --}
\DeclareMathOperator{\vol}{vol}

\begin{document}
\maketitle

\begin{abstract}
We investigate different concentration-compactness phenomena related to the $Q$-curvature in arbitrary even dimension. We first treat the case of an open domain in $\R{2m}$, then that of a closed manifold and, finally, the particular case of the sphere $S^{2m}$. In all cases we allow the sign of the $Q$-curvature to vary, and show that in the case of a closed manifold, contrary to the case of open domains in $\R{2m}$, concentration phenomena can occur only at points of positive $Q$-curvature. As a consequence, on a locally conformally flat manifold of non-positive Euler characteristic we always have compactness.
\end{abstract}

\section{Introduction and statement of the main results}

Before stating our results, we recall a few facts concerning the Paneitz operator $P^{2m}_g$ and the $Q$-curvature $Q^{2m}_g$ on a $2m$-dimensional smooth Riemannian manifold $(M,g)$. Introduced in \cite{BO}, \cite{pan}, \cite{bra} and \cite{GJMS}, the Paneitz operator and the $Q$-curvature are the higher order equivalents of the Laplace-Beltrami operator and the Gaussian curvature respectively ($P^2_g=-\Delta_g$ and $Q^2_g=K_g$), and they now play a central role in modern conformal geometry.
For their definitions and more related information we refer to \cite{cha}. Here we only recall a few properties which shall be used later.
First of all we have the Gauss formula, describing how the $Q$-curvature changes under a conformal change of metric:
\begin{equation}\label{gauss}
P^{2m}_g u + Q^{2m}_g=Q^{2m}_{g_u}e^{2mu},
\end{equation}
where $g_u:=e^{2u}g$, and $u\in C^\infty(M)$ is arbitrary.
Then, we have the conformal invariance of the total $Q$-curvature, when $M$ is closed:
\begin{equation}\label{total}
\int_M Q_{g_u}^{2m}\mathrm{dvol}_{g_{u}}=\int_M Q_g^{2m} \mathrm{dvol}_g.
\end{equation}
Finally, assuming $(M,g)$ closed and locally conformally flat , we have the Gauss-Bonnet-Chern formula (see e.g. \cite{che}, \cite{cha}):
\begin{equation}\label{GBC}
\int_M Q^{2m}_g \mathrm{dvol}_g =\frac{\Lambda_1}{2}\chi(M),
\end{equation}
where $\chi(M)$ is the Euler-Poincar\'e characteristic of $M$ and
\begin{equation}\label{lambda1}
\Lambda_1:=\int_{S^{2m}} Q_{g_{S^{2m}}}\mathrm{dvol}_{g_{S^{2m}}}=(2m-1)!|S^{2m}|
\end{equation}
is a constant which we shall meet often in the sequel. In the $4$-dimensional case, if $(M,g)$ is not locally conformally flat, we have
\begin{equation}\label{GBC4}
\int_M\bigg(Q^4_g+\frac{|W_g|^2}{4}\bigg)\dvol_g=8\pi^2\chi(M),
\end{equation}
where $W_g$ is the Weyl tensor. Recently S. Alexakis \cite{ale2} (see also \cite{ale}) proved an analogous to \eqref{GBC4} for $m\geq 3$:
\begin{equation}\label{GBC5}
\int_M\bigg(Q^{2m}_g+W\bigg)\dvol_g=\frac{\Lambda_1}{2}\chi(M),
\end{equation}
where $W$ is a local conformal invariant involving the Weyl tensor and its covariant derivatives.

\medskip

We can now state the main problem treated in this paper. Given a $2m$-dimensional Riemannian manifold $(M,g)$, consider a converging sequence of functions $Q_k \to Q_0$ in $C^0(M)$, and let $g_k:=e^{2u_k}g$ be conformal metrics satisfying $Q^{2m}_{g_k}=Q_k$. In view of \eqref{gauss}, the $u_k$'s satisfy the following elliptic equation of order $2m$ with critical exponential non-linearity
\begin{equation}\label{presc}
P^{2m}_g u_k +Q^{2m}_g=Q_ke^{2mu_k}.
\end{equation} Assume further that there is a constant $C>0$ such that
\begin{equation}\label{volume}
\mathrm{vol}(g_k)=\int_{M}e^{2mu_k}\dvol_{g}\leq C\quad \textrm{for all }k.
\end{equation}
What can be said about the compactness properties of the sequence $(u_k)$? 

In general non-compactness has to be expected, at least as a consequence of the non-compactness of the M\"obius group on $\R{2m}$ or $S^{2m}$. For instance, for every $\lambda>0$ and $x_0\in \R{2m}$, the metric on $\R{2m}$ given by $g_u:=e^{2u}g_{\R{2m}}$, $u(x):=\log\frac{2\lambda}{1+\lambda^2|x-x_0|^2}$, satisfies $Q^{2m}_{g_u}\equiv (2m-1)!$.

\medskip

We start by considering the case when $(M,g)$ is an open domain $\Omega\subset\R{2m}$ with Euclidean metric $g_{\R{2m}}$. Since $P_{g_{\R{2m}}}=(-\Delta)^m$ and $Q_{g_{\R{2m}}}\equiv 0$, Equation \eqref{presc} reduces to $(-\Delta)^m u_k=Q_ke^{2mu_k}$. The compactness properties of this equation were studied in dimension $2$ by Br\'ezis and Merle \cite{BM}. They proved that if $Q_k\geq 0$, $\|Q_k\|_{L^\infty}\leq C$ and $\|e^{2u_k}\|_{L^1}\leq C$, then up to selecting a subsequence, one of the following is true:
\begin{itemize}
\item[(i)] $(u_k)$ is bounded in $L^\infty_{\loc}(\Omega)$.
\item[(ii)] $u_k\to-\infty$ locally uniformly in $\Omega$.
\item[(iii)] There is a finite set $S=\{x^{(i)};i=1,\ldots,I\}\subset \Omega$ such that $u_k\to -\infty$ locally uniformly in $\Omega\backslash S$. Moreover $Q_k e^{2u_k}\rightharpoonup \sum_{i=1}^I \beta_i \delta_{x^{(i)}}$ weakly in the sense of measures, where $\beta_i\geq 2\pi$ for every $1\leq i\leq I$.
\end{itemize}
Subsequently, Li and Shafrir \cite{LS} proved that in case (iii) $\beta_i\in 4\pi\mathbb{N}$ for every $1\leq i\leq I$.

Adimurthi, Robert and Struwe \cite{ARS} studied the case of dimension $4$ ($m=2$). As they showed, the situation is more subtle because the blow-up set (the set of points $x$ such that $u_k(x)\to\infty$ as $k\to\infty$) can have dimension up to $3$ (in contrast to the finite blow-up set $S$ in dimension $2$). Moreover, as a consequence of a result of Chang and Chen \cite{CC}, quantization in the sense of Li-Shafrir does not hold anymore, see also \cite{rob1}, \cite{rob2}.

In the following theorem we extend the result of \cite{ARS} to arbitrary even dimension (see also Proposition \ref{propblowup} below). The function $a_k$ in \eqref{eqliou} has no geometric meaning, and one can take $a_k\equiv 1$ at first. On the other hand, one can also apply Theorem \ref{main} to non-geometric situations, by allowing $a_k\not\equiv 1$, see \cite{mar3}.

\begin{trm}\label{main} Let $\Omega$ be a domain in $\R{2m}$, $m>1$, and let $(u_k)_{k\in \mathbb{N}}$ be a sequence of functions satisfying
\begin{equation}\label{eqliou}
(-\Delta)^m u_k=Q_k e^{2m a_ku_k},
\end{equation}
where $a_k,Q_0\in C^0(\Omega)$, $Q_0$ is bounded, and $Q_k\rightarrow Q_0$, $a_k\to 1$ locally uniformly. Assume that
\begin{equation}\label{vk}
\int_{\Omega} e^{2m a_k u_k}dx\leq C,
\end{equation}
for all $k$ and define the finite (possibly empty) set 
$$S_1:=\bigg\{x\in\Omega : \lim_{r\to 0^+}\limsup_{k\to\infty}\int_{B_r(x)}|Q_k|e^{2ma_ku_k}dy\geq \frac{\Lambda_1}{2} \bigg\}=\big\{x^{(i)}: 1\leq i\leq I\big\},$$
where $\Lambda_1$ is as in \eqref{lambda1}. Then one of the following is true.
\begin{enumerate}
\item[(i)] For every $0\leq\alpha<1$, a subsequence converges in $C^{2m-1,\alpha}_{\loc}(\Omega\backslash S_1)$.

\item[(ii)] There exist a subsequence, still denoted by $(u_k)$, and a closed nowhere dense set $S_0$ of Hausdorff dimension at most $2m-1$ such that, letting $S=S_0\cup S_1,$
we have $u_k\rightarrow -\infty$  locally uniformly in $\Omega\bs S$ as $k\rightarrow \infty.$ Moreover there is a sequence of numbers $\beta_k\rightarrow \infty$ such that
$$\frac{u_k}{\beta_k}\rightarrow \varphi \textrm{ in } C^{2m-1,\alpha}_{\loc}(\Omega\backslash S),\quad 0\leq \alpha<1,$$
where $\varphi\in C^\infty(\Omega\bs S_1)$, $S_0=\{x\in\Omega:\varphi(x)=0\}$, and
$$(-\Delta)^m\varphi\equiv 0,\quad \varphi\leq 0,\quad \varphi \not\equiv 0\quad \textrm{in }\Omega\backslash S_1.$$
\end{enumerate}
If $S_1\neq \emptyset$ and $Q_0(x^{(i)})>0$ for some $1\leq i\leq I$, then case (ii) occurs.
\end{trm}

We recently proved (see \cite{mar2}) the existence of solutions to the equation $(-\Delta)^m u=Q e^{2mu}$ on $\R{2m}$ with $ Q<0$ constant and $e^{2mu}\in L^1(\R{2m})$, for $m>1$. Scaling any such solution we find a sequence of solutions $u_k(x):=u(kx)+\log k$ concentrating at a point of negative $Q$-curvature. For $m=1$ that is not possible.

\medskip

On a closed manifold things are different in several respects. Under the assumption (which we always make) that $\ker P^{2m}_g$ contains only constant functions, quantization of the total $Q$-curvature in the sense of Li-Shafrir (see \eqref{intQg} below) holds, as proved in dimension $4$ by Druet and Robert \cite{DR} and Malchiodi \cite{mal}, and in arbitrary dimension by Ndiaye \cite{ndi}. Moreover the concentration set is finite. In \cite{DR}, however, it is assumed that the $Q$-curvatures are positive, while in \cite{mal} and \cite{ndi}, a slightly different equation is studied ($P^{2m}_gu_k+ Q_k=h_ke^{2mu_k}$, with $h_k$ constant and $Q_k$ prescribed), for which the negative case is simpler. With the help of results from our recent work \cite{mar2} and a technique of Robert and Struwe \cite{RS}, we can allow the prescribed $Q$-curvatures to have varying signs and, contrary to the case of an open domain in $\R{2m}$, we can rule out concentration at points of negative $Q$-curvature.

\begin{trm}\label{trmquant} Let $(M,g)$ be a $2m$-dimensional closed Riemannian manifold, such that $\ker P_g=\{constants\}$, and let $(u_k)$ be a sequence of solutions to \eqref{presc}, \eqref{volume} 
where the $Q_k$'s and $Q_0$ are given $C^1$ functions and $Q_k\to Q_0$ in $C^1(M)$. Let $\Lambda_1$ be as in \eqref{lambda1}. Then one of the following is true.
\begin{enumerate}
\item[(i)] For every $0\leq \alpha<1$, a subsequence converges in $C^{2m-1,\alpha}(M)$. 
\item[(ii)] There exists a finite (possibly empty) set $S_1=\{x^{(i)}:1\leq i\leq I\}$ such that $Q_0(x^{(i)})>0$ for $1\leq i\leq I$ and, up to taking a subsequence, $u_k\to -\infty$ locally uniformly on $(M\bs S_1)$. Moreover
\begin{equation}\label{QgL}
Q_k e^{2mu_k}\dvol_g \rightharpoonup\sum_{i=1}^I \Lambda_1\delta_{x^{(i)}}
\end{equation}
in the sense of measures; then \eqref{total} gives
\begin{equation}\label{intQg}
\int_M Q_g \mathrm{dvol}_g = I\Lambda_1.
\end{equation}
Finally, $S_1=\emptyset$ if and only if $\mathrm{vol}(g_k)\to 0$. 
\end{enumerate}
\end{trm}

An immediate consequence of Theorem \ref{trmquant} (Identity \eqref{intQg} in particular) and the Gauss-Bonnet-Chern formulas \eqref{GBC} and \eqref{GBC4}, is the following compactness result:

\begin{cor}\label{compa} Under the hypothesis of Theorem \ref{trmquant} assume that either
\begin{enumerate}
\item $\chi(M)\leq 0$ and $\dim M\in\{2,4\}$, or
\item $\chi(M)\leq 0$, $\dim M\geq 6$ and $(M,g)$ is locally conformally flat,
\end{enumerate}
and that $\mathrm{vol}(g_k)\not\to 0$. Then (i) in Theorem \ref{trmquant} occurs.
\end{cor}

It is not clear whether the hypothesis that $(M,g)$ be locally conformally flat when $\dim M\geq 6$ is necessary in Corollary \ref{compa}. For instance, we could drop it if we knew that $W\geq 0$ in \eqref{GBC5}, in analogy with \eqref{GBC4}.

\medskip

Contrary to what happens for the Yamabe equation (see \cite{dru1}, \cite{dru2}, \cite{DH} and \cite{DHR}), the concentration points of $S$ in Theorem \ref{trmquant} are isolated, as already proved in \cite{DR} in dimension 4. In fact, a priori one could expect to have
\begin{equation}\label{QgL2}
Q_k e^{2mu_k}\dvol_g \rightharpoonup\sum_{i=1}^I L_i\Lambda_1\delta_{x^{(i)}},\quad \text{for some }L_i\in \mathbb{N}\backslash \{0\},
\end{equation}
instead of \eqref{QgL}. The compactness of $M$ is again a crucial ingredient here; indeed X. Chen \cite{ch} showed that on $\R{2}$ (where quantization holds, as already discussed) one can have \eqref{QgL2} with $L_i>1$.

\medskip

Theorems \ref{main} and \ref{trmquant} will be proven in Sections \ref{sectionmain} and \ref{sectionclosed} respectively.
In Section \ref{sectionsphere} we also consider the special case when $M=S^{2m}$.
In the proofs of the above theorems we use techniques and ideas from several of the cited papers, particularly from \cite{ARS}, \cite{BM}, \cite{DR}, \cite{mal}, \cite{MS} and \cite{RS}.
In the following, the letter $C$ denotes a generic positive constant, which may change from line to line and even within the same line.

\medskip

I'm grateful to Prof. Michael Struwe for many stimulating discussions.

\section{The case of an open domain in $\R{2m}$}\label{sectionmain}

An important tool in the proof of Theorem \ref{main} is the following estimate, proved by Br\'ezis and Merle \cite{BM} in dimension $2$. For the proof in arbitrary dimension see \cite{mar1}. Notice the role played by the constant $\gamma_m:=\frac{\Lambda_1}{2}$, which satisfies
\begin{equation}\label{fund}
(-\Delta)^m\Big(-\frac{1}{\gamma_m}\log|x|\Big)=\delta_0\quad \textrm{in }\R{2m}.
\end{equation}

\begin{trm}\label{a2m} Let $f\in L^1(B_R(x_0))$, $B_R(x_0)\subset\R{2m}$, and let $v$ solve
$$\left\{
\begin{array}{ll}
(-\Delta)^m v=f &\textrm{in } B_R(x_0),\\
v=\Delta v=\ldots=\Delta^{m-1}v=0& \textrm{on }\partial B_R(x_0).
\end{array}
\right.
$$
Then, for any $p\in\Big(0,\frac{\gamma_m}{\|f\|_{L^1(B_R(x_0))}}\Big)$, we have $e^{2m p|v|}\in L^1(B_R(x_0))$ and
$$\int_{B_R(x_0)}e^{2m p |v|}dx\leq C(p)R^{2m}.$$
\end{trm}

\begin{lemma}\label{propmu} Let $f\in L^1(\Omega)\cap L^p_{\loc}(\Omega\bs S_1)$ for some $p>1$, where $\Omega\subset\R{2m}$ and $S_1\subset \Omega$ is a finite set. Assume that
$$\left\{
\begin{array}{ll}
(-\Delta)^m u=f & \textrm{in }\Omega\\
\Delta^j u=0 &\textrm{on }\partial\Omega \textrm{ for } 0\leq j\leq m-1.
\end{array}
\right.$$
Then $u$ is bounded in $W^{2m,p}_{\loc}(\Omega\backslash S_1)$; more precisely, for any $\overline{B_{4R}(x_0)}\subset (\Omega\bs S_1)$, there is a constant $C$ independent of $f$ such that
\begin{equation}\label{21*}
\|u\|_{W^{2m,p}(B_R(x_0))}\leq C (\|f\|_{L^p(B_{4R}(x_0))}+\|f\|_{L^1(\Omega)}).
\end{equation}
\end{lemma}
The proof of Lemma \ref{propmu} is given in the appendix.

\medskip

\noindent\emph{Proof of Theorem \ref{main}.} We closely follow \cite{ARS}. Let $S_1$ be defined as in the statement of the Theorem. Clearly \eqref{vk} implies that $S_1=\{x^{(i)}\in\Omega :1\leq i\leq I\}$ is finite.
Given $x_0\in\Omega\backslash S_1$, we have, for some $0<R<\dist(x_0,\partial\Omega),$
\begin{equation}\label{limsup}
\alpha:=\limsup_{k\rightarrow \infty}\int_{B_R(x_0)}|Q_k|e^{2m a_k u_k}dx< \gamma_m.
\end{equation}
For such $x_0$ and $R$ write $u_k=v_k+h_k$ in $B_R(x_0)$, where

$$
\left\{
\begin{array}{ll}
(-\Delta)^mv_k=Q_ke^{2m a_k u_k}&\textrm{in } B_R(x_0)\\
v_k=\Delta v_k=\ldots=\Delta^{m-1}v_k=0&\textrm{on }\partial B_R(x_0)
\end{array}
\right.
$$
and $(-\Delta)^mh_k=0$. 
Set $h_k^+:=\chi_{\{h_k\geq 0\}}h_k$, $h_k^-:=h_k-h_k^+$. Since $h_k^+\leq u_k^+ +|v_k|$, we have
$$\|h_k^+\|_{L^1(B_R(x_0))}\leq \|u_k^+\|_{L^1(B_R(x_0))}+\|v_k\|_{L^1(B_R(x_0))}.$$
Observe that, for $k$ large enough $mu_k^+\leq 2ma_ku_k^+\leq e^{2m a_k u_k}$ on $B_R(x_0)$, hence \eqref{vk} implies
$$\int_{B_R(x_0)}u_k^+dx\leq C\int_{B_R(x_0)}e^{2m a_k u_k}dx\leq C.$$
As for $v_k$, observe that $1<\frac{\gamma_m}{\alpha}$, hence by Theorem \ref{a2m}
$$\int_{B_R(x_0)}2m|v_k|dx\leq\int_{B_R(x_0)}e^{2m|v_k|}dx\leq CR^{2m},$$
with $C$ depending on $\alpha$ and not on $k$. Hence
\begin{equation}\label{hk+}
\|h_k^+\|_{L^1(B_R(x_0))}\leq C.
\end{equation}

We distinguish 2 cases.

\medskip

\noindent\emph{Case 1.} Suppose that $\|h_k\|_{L^1(B_{R/2})(x_0)}\leq C$ uniformly in $k$. Then by Proposition \ref{c2m} we have that $h_k$ is equibounded in $C^\ell(B_{R/8}(x_0))$ for every $\ell\geq 0$. Moreover, by Pizzetti's formula (Identity \eqref{pizzetti} in the appendix) and \eqref{hk+},
\begin{eqnarray*}
\Intm_{B_R(x_0)}|h_k(x)|dx&=&2\Intm_{B_R(x_0)} h_k^+(x)dx-\Intm_{B_R(x_0)} h_k(x)dx\\
&\leq& C -\Intm_{B_R(x_0)} h_k(x)dx\\
&=&C-\sum_{i=0}^{m-1}c_iR^{2i}\Delta^ih_k(x_0) \leq C.
\end{eqnarray*}
Hence we can apply Proposition \ref{c2m} locally on all of $B_R(x_0)$ and obtain bounds for $(h_k)$ in $C^\ell_{\loc}(B_R(x_0))$ for any $\ell\geq 0$.

Fix $p\in (1,\gamma_m/\alpha)$. By Theorem \ref{a2m} $\|e^{2m|v_k|}\|_{L^p(B_R(x_0))}\leq C(p)$, hence, using that $a_k\to 1$ uniformly on $B_R(x_0)$, we infer

\begin{equation}\label{vk2}
\|(-\Delta)^mv_k\|_{L^p(B)}=\|(Q_ke^{2m a_k h_k})e^{2m a_k v_k}\|_{L^p(B)}\leq C(B,p)
\end{equation}
for every ball $B \subset\subset B_R(x_0)$ and for $k$ large enough. In addition $\|v_k\|_{L^1(B_R(x_0))}\leq C$, hence by elliptic estimates,
$$\|v_k\|_{W^{2m,p}(B)}\leq C(B,p)\quad \text{for every ball } B\subset\subset B_R(x_0).$$ By the immersion $W^{2m,p}\hookrightarrow C^{0,\alpha}$, $(v_k)$, is bounded in $C^{0,\alpha}_{\loc}(B_R(x_0))$, for some $\alpha>0$.
Going back to \eqref{vk2}, we now see that $\Delta^m v_k$ is uniformly bounded in $L^\infty_{\loc}(B_R(x_0))$, hence
$$\|v_k\|_{W^{2m,p}(B)}\leq C(B,p)$$
for every $p>1$, $B\subset\subset B_R(x_0)$, and by the immersion $W^{2m,p}\hookrightarrow C^{2m-1,\alpha}$ we obtain that $(v_k)$, hence $(u_k)$, is bounded in $C^{2m-1,\alpha}_{\loc}(B_R(x_0))$.

\medskip

\noindent\emph{Case 2.} Assume that $\|h_k\|_{L^1(B_{R/2}(x_0))}=:\beta_k\rightarrow \infty$ as $k\rightarrow \infty$. Set $\varphi_k:=\frac{h_k}{\beta_k}$, so that
\begin{enumerate}
\item $\Delta^m \varphi_k=0,$
\item  $\|\varphi_k\|_{L^1(B_{R/2}(x_0))}=1$,
\item $\|\varphi_k^+\|_{L^1(B_R(x_0))}\rightarrow 0$ by \eqref{hk+}.
\end{enumerate}
As above we have that $\varphi_k$ is bounded in $C^\ell_{\loc}(B_R(x_0))$ for every $\ell\geq 0$, hence a subsequence converges in $C^{2m}_{\loc}(B_R(x_0))$ to a function $\varphi$, with
\begin{enumerate}
\item $\Delta^m \varphi=0,$
\item  $\|\varphi\|_{L^1(B_{R/2}(x_0))}=1$,
\item $\|\varphi^+\|_{L^1(B_R(x_0))}=0$, hence $\varphi\leq 0$.
\end{enumerate}

Let us define $S_0=\{x\in B_R(x_0):\varphi(x)=0\}$. Take $x\in S_0$; then by \eqref{pizzetti}, $\Delta \varphi(x),\ldots,\Delta^{m-1}\varphi(x)$ cannot all vanish, unless $\varphi\equiv 0$ on $B_\rho(x)\subset B_R(x_0)$ for some $\rho>0$, but then by analyticity, we would have $\varphi\equiv 0,$ contradiction. Hence there exists $j$ with $1\leq j\leq 2m-3$ such that
$$D^j\varphi(x)=0,\quad D^{j+1}\varphi(x)\neq 0,$$
i.e.
$$S_0\subset \bigcup_{j=1}^{2m-3}\{x\in B_R(x_0):D^j\varphi(x)=0,D^{j+1}\varphi(x)\neq 0\}.$$
Therefore $S_0$ is $(2m-1)$-rectifiable.
Since $\varphi<0$ on $B_R(x_0)\backslash S_0$, we infer
$$h_k=\beta_k\varphi_k\rightarrow -\infty, \quad e^{2m a_k h_k}\rightarrow 0$$
locally uniformly on $B_R(x_0)\backslash S_0$. Then, as before, from
$$(-\Delta)^m v_k=(Q_ke^{2m a_k h_k})(e^{2m a_k v_k}),$$
we have that $v_k$ is bounded in $C^{2m-1,\alpha}_{\loc}(\Omega\backslash S_0)$. Then $u_k=h_k+v_k\rightarrow -\infty$ uniformly locally away from $S_0$.

\medskip

Since Case 1 and Case 2 are mutually exclusive, covering $\Omega\backslash S_1$ with balls, we obtain that either a subsequence $u_k$ is bounded in $C^{2m-1,\alpha}_{\loc}(\Omega\backslash S_1)$, or a subsequence $u_k\rightarrow -\infty$ locally uniformly on $\Omega\backslash (S_0\cup S_1)$. In this latter case, the behavior described in case (ii) of the theorem occurs. Indeed fix any $B_R(x_0)\subset\Omega\backslash S_1$ and take $\beta_k$ as above. Then, on a ball $B_\rho (y_0)\subset \Omega\backslash S_1$, we can wrie $u_k=\tilde v_k+\tilde h_k$ as above, where $\tilde h_k\to-\infty$ locally uniformly away from a rectifiable set $S_0$ of dimension at most $(2m-1)$, $\frac{\tilde h_k}{\tilde \beta_k}\to \tilde \varphi$ ,  where $\tilde \beta_k=\|\tilde h_k\|_{L^1(B_{\rho/2}(y))}$, and $\tilde v_k$ is bounded in $C^{2m-1,\alpha}_{\loc}(B_\rho(y_0))$. Then $\frac{\tilde v_k}{\beta_k}\to 0$ in $C^{2m-1,\alpha}_{\loc}(B_\rho(y_0))$, and we have that either
\begin{itemize}
\item[(a)] $\frac{\tilde h_k}{\beta_k}$ and $\frac{u_k}{\beta_k}$ are bounded in $C^{2m-1,\alpha}_{\loc}(B_\rho(y_0))$, or
\item[(b)] $\frac{\tilde h_k}{\beta_k}$ and $\frac{u_k}{\beta_k}$ go to $-\infty$ locally uniformly away from $S_0$.
\end{itemize}
Since the 2 cases are mutually exclusive, and on $B_R(x_0)$ case (a) occurs, upon covering $\Omega\backslash S_1$ with a sequence of balls,  we obtain the desired behavior for $\frac{u_k}{\beta_k}$.

\medskip

We now show that if $I\geq 1$ and $Q_0(x^{(i)})>0$ for some $1\leq i\leq I$, then Case 2 occurs. Assume by contradiction that $Q_0(x_0)>0$ for some $x_0\in S_1$ and Case 1 occurs, i.e. $(u_k)$ is bounded in $C^{2m-1,\alpha}_{\loc}(\Omega\bs S_1)$, so that $f_k:=Q_ke^{2m a_k u_k}$ is bounded in $L^\infty_{\loc}(\Omega\backslash S_1)$. Then there exists a finite signed measure $\mu$ on $\Omega$, with $\mu\in L^\infty_{\loc}(\Omega\bs S_1)$ such that
\begin{eqnarray*}
f_k &\rightharpoonup& \mu \quad \textrm{as measures}\\
f_k &\rightharpoonup & \mu \quad \textrm{in } L^p_{\loc}(\Omega\bs S_1)\textrm{ for } 1\leq  p < \infty.
\end{eqnarray*}
Let us take $R>0$ such that $\overline{B_R(x_0)}\subset \Omega$, $B_R(x_0)\cap S_1=\{x_0\}$ and $Q_0>0$ on $B_R(x_0)$. By our assumption, 
\begin{equation}\label{geqC}
(-\Delta)^ju_k\geq -C,\quad\textrm{on }\partial B_R(x_0) \textrm{ for }0\leq j\leq m-1.\end{equation}
Let $z_k$ be the solution to
$$
\left\{
\begin{array}{ll}
(-\Delta)^m z_k= Q_k e^{2m a_k u_k}&\textrm{in }B_R(x_0)\\
z_k=\Delta z_k=\ldots=\Delta^{m-1} z_k=0&\textrm{on }\partial B_R(x_0).
\end{array}
\right.
$$
By Proposition \ref{propmax}, and \eqref{geqC}
\begin{equation}\label{ukc}
u_k\geq z_k- C.
\end{equation}
By Lemma \ref{propmu}, up to a subsequence, $z_k\rightarrow z$ in $C^{2m-1,\alpha}_{\loc}(B_R(x_0)\bs\{x_0\})$, where
$$
\left\{
\begin{array}{ll}
(-\Delta)^m z= \mu &\textrm{in }B_R(x_0)\\
z=\Delta z=\ldots=\Delta^{m-1} z=0&\textrm{on }\partial B_R(x_0).
\end{array}
\right.
$$
Since $Q_0(x_0)>0$, we have $\mu\geq \gamma_m \delta_{x_0}=(-\Delta)^m\ln \frac{1}{|x-x_0|}$, and Proposition \ref{propmax} applied to the function $z(x)-\ln\frac{1}{|x-x_0|}$ implies
$$z(x)\geq \ln \frac{1}{|x-x_0|} -C,$$
hence
$$\int_{B_R(x_0)}e^{2m z}dx\geq \frac{1}{C}\int_{B_R(x_0)}\frac{1}{|x-x_0|^{2m}}dx=+\infty.$$
Then \eqref{ukc} and Fatou's lemma imply
\begin{eqnarray}
\liminf_{k\to\infty}\int_{B_R(x_0)} e^{2m a_k u_k}dx &\geq& \int_{B_R(x_0)}\liminf_{k\to\infty}e^{2m a_k u_k}dx\nonumber \\
&\geq& \frac{1}{C} \int_{B_R(x_0)}\liminf_{k\to\infty} e^{2m a_k z_k}dx\label{fatou}\\
& \geq& \frac{1}{C}\int_{B_R(x_0)}e^{2m z}dx=+\infty,\nonumber
\end{eqnarray}
contradicting \eqref{vk}.
\hfill$\square$

\medskip

The following proposition gives a general procedure to rescale at points where $u_k$ goes to infinity. 

\begin{prop}\label{propblowup}
In the hypothesis of Theorem \ref{main}, assume that $a_k\equiv 1$ for every $k$ and that case (ii) occurs. Then, for every $x_0\in S$ such that $\sup_{B_R(x_0)} u_k\rightarrow \infty$ for every $0<R<\dist(x_0,\partial \Omega)$ as $k\rightarrow \infty$, there exist points $x_k\rightarrow x_0$ and positive numbers $r_k\rightarrow 0$ such that
\begin{equation}\label{scal}
v_k(x):=u_k(x_k+r_k x)+\ln r_k\leq 0\leq \ln 2+v_k(0),
\end{equation}
and as $k\rightarrow \infty$ either a subsequence $v_k\rightarrow v$ in $C^{2m-1,\alpha}_{\loc}(\R{2m})$, where
$$(-\Delta)^m v=Q_0(x_0)e^{2mv},$$
or $v_k\rightarrow -\infty$ almost everywhere and there are positive numbers $\gamma_k\to+\infty$ such that
$$\frac{v_k}{\gamma_k}\to p \quad \textrm{in } C^{2m-1,\alpha}_{\loc}(\R{2m}),$$
where $p$ is a polynomial on even degree at most $2m-2$.
\end{prop}

\begin{proof}
Following \cite{ARS}, take $x_0$ such that $\sup_{B_R(x_0)} u_k \rightarrow +\infty$ for every $R$ and select, for $R<\dist (x_0,\partial\Omega)$, $0\leq r_k<R$ and $x_k\in \overline{B_{r_k}(x_0)}$ such that
$$(R-r_k)e^{u_k(x_k)}=(R-r_k)\sup_{\overline{B_{r_k}(x_0)}} e^{u_k}=\max_{0\leq r<R}\Big((R-r)\sup_{\overline{B_{r}(x_0)}} e^{u_k}\Big)=:L_k.$$
Then $L_k\rightarrow +\infty$ and $s_k:=\frac{R-r_k}{2L_k}\rightarrow 0$ as $k\rightarrow \infty$, and
$$v_k(x):=u_k(x_k+s_k  x)+\ln s_k\leq 0\quad\textrm{in }B_{L_k}(0)$$
satisfies
$$(-\Delta)^m v_k=\widetilde Q_k e^{2mv_k},\quad \widetilde Q_k(x):=Q_k(x_k+s_kx),$$
and
$$\int_{B_{L_k}(0)} \widetilde Q_k e^{2mv_k}dx=\int_{B_{\frac{1}{2}(R-r_k)}(x_k)} Q_k e^{2m u_k}dx\leq C.$$
We can now apply the first part of the theorem to the functions $v_k$, observing that there are no concentration points ($S_1=\emptyset$), since $v_k\leq 0$, and using Theorem \ref{trmliou} to characterize the function $p$.
\end{proof}

\section{The case of a closed manifold}\label{sectionclosed}

To prove Theorem \ref{trmquant} we assume that $\sup_M u_k\to\infty$ and we blow up at $I$ suitably chosen sequences of points $x_{i,k}\to x^{(i)}$ with $u_k(x_{i,k})\to\infty$ as $k\to\infty$, $1\leq i\leq I$. We call the $x^{(i)}$'s concentration points. Then we show the following:

\begin{enumerate}
\item[(i)] If $x^{(i)}$ is a concentration point, then $Q_0(x^{(i)})>0$.
\item[(ii)] The profile of the $u_k$'s at any concentration point is the function $\eta_0$ defined in \eqref{etak} below, hence it carries the fixed amount of energy $\Lambda_1$, see \eqref{etak2}.
\item[(iii)] $u_k\to -\infty$ locally uniformly in $M\backslash\{x^{(i)}:1\leq i\leq I\}$.
\item[(iv)] The \emph{neck energy} vanishes in the sense of \eqref{vanish} below, hence in the limit only the energy of the profiles at the concentration points appears.
\end{enumerate}

Parts (i) and (ii) (Proposition \ref{blow}) follow from Lemma \ref{nabla} below and the classification results of \cite{mar1} (or \cite{xu}) and \cite{mar2}.
For parts (iii) and (iv) we adapt a technique of \cite{DR}, see also also \cite{mal}, \cite{ndi} for a different approach.

\medskip

The following lemma (compare \cite[Lemma 2.3]{mal}) is important, because its failure in the non-compact case is responsible for the rich concentration-compactness behavior in Theorem \ref{main}. Its proof relies on the existence and on basic properties of the Green function for the Paneitz operator $P_g^{2m}$, as proven in \cite[Lemma 2.1]{ndi} (here we need the hypothesis $\ker P^{2m}_g=\{constants\}$).

\begin{lemma}\label{nabla} Let $(u_k)$ be a sequence of functions on $(M,g)$ satisfying \eqref{presc} and \eqref{volume}. Then for $\ell=1,\ldots, 2m-1$, we have
$$\int_{B_r(x)}|\nabla^\ell u_k|^p \dvol_g\leq C(p) r^{2m-\ell p}, \quad 1\leq p<\frac{2m}{\ell},$$
for every $x\in M$, $0<r<\ring$ and for every $k$, where $\ring$ is the injectivity radius of $(M,g)$.
\end{lemma}

\begin{proof} Set $f_k:=Q_ke^{2mu_k}-Q^{2m}_{g}$, which is bounded in $L^1(M)$ thanks to \eqref{volume}. Let $G_\xi$ be the Green's function for $P^{2m}_g$ on $(M,g)$ such that
\begin{equation}\label{rapp}
u_k(\xi)=\Intm_M u_k \dvol_g +\int\limits_{M}G_\xi(y)f_k(y)\dvol_g(y).
\end{equation}
For $x,\xi\in M$, $x\neq \xi$, \cite[Lemma 2.1]{ndi} implies
\begin{equation}\label{nablag}
|\nabla^\ell_\xi G_\xi(x)|\leq \frac{C}{\dist(x,\xi)^\ell},\quad 1\leq \ell\leq 2m-1.
\end{equation}
Then, differentiating \eqref{rapp} and using \eqref{nablag} and Jensen's inequality, we get

\begin{eqnarray*}
|\nabla^\ell u_k(\xi)|^p&\leq&C\bigg(\int_{M} \frac{1}{\dist(\xi,y)^\ell}|f_k(y)|\dvol_g(y)\bigg)^p\\
&\leq&C \int_{M}\bigg(\frac{\|f_k\|_{L^1(M)}}{\dist(\xi,y)^{\ell}}\bigg)^p \frac{|f_k(y)|}{\|f_k\|_{L^1(M)}}\dvol_g(y).
\end{eqnarray*}
From Fubini's theorem we then conclude
\begin{eqnarray*}
\int_{B_r(x)}|\nabla^\ell u_k(\xi)|^p \dvol_g(\xi)&\leq&C\|f_k\|_{L^1(M)}^p\sup_{y\in M}\int_{B_r(x)}\frac{1}{\dist(\xi,y)^{\ell p}}\dvol_g(\xi)\\
&\leq&C r^{2m-\ell p}.
\end{eqnarray*}
\end{proof}

Let $\exp_x:T_xM \cong \R{2m}\to M$ denote the exponential map at $x$.

\begin{prop}\label{blow} Let $(u_k)$ be a sequence of solutions to \eqref{presc}, \eqref{volume} with $\max u_k\to \infty$ as $k\to \infty$. Choose points $x_k\to x_0\in M$ (up to a subsequence) such that $u_k(x_k)=\max_M u_k$. Then $Q_0(x_0)>0$ and, setting
\begin{equation}\label{defmuk}
\mu_k:=2\bigg(\frac{(2m-1)!}{Q_0(x_0)}\bigg)^\frac{1}{2m}e^{-u_k(x_k)}\end{equation}
we find that the functions $\eta_k:B_\frac{\ring}{\mu_k}\subset\R{2m}\to\R{}$, given by
$$\eta_k(y):=u_k(\exp_{x_k}(\mu_k y))+\log \mu_k-\frac{1}{2m}\log\frac{(2m-1)!}{Q_0(x_0)},$$
converge up to a subsequence to $\eta_0(y)=\ln\frac{2}{1+|y|^2}$ in $C^{2m-1,\alpha}_{\loc}(\R{2m})$.
Moreover
\begin{equation}\label{energia}
\lim_{R\to+\infty}\lim_{k\to\infty}\int_{B_{R \mu_k}(x_k)}Q_k e^{2mu_k}\dvol_g=\Lambda_1.
\end{equation}
\end{prop}

\noindent\emph{Remark.} The function
\begin{equation}\label{etak}
\eta_0(x):=\log\frac{2}{1+|x|^2}
\end{equation}
satisfies $(-\Delta)^m \eta_0 =(2m-1)!e^{2m\eta_0}$, which is \eqref{eqliou} with $Q_k\equiv (2m-1)!$ and $a_k\equiv 1$. In fact $\eta_0$ has a remarkable geometric interpretation: If $\pi:S^{2m}\to \R{2m}$ is the stereographic projection, then
\begin{equation}\label{proj}
e^{2\eta_0}g_{\R{2m}}=(\pi^{-1})^* g_{S^{2m}},
\end{equation}
where $g_{S^{2m}}$ is the round metric on $S^{2m}$.
Then \eqref{proj} implies
\begin{equation}\label{etak2}
(2m-1)!\int_{\R{2m}}e^{2m \eta_0}dx=\int_{S^{2m}}Q_{S^{2m}}\mathrm{dvol}_{g_{S^{2m}}}=(2m-1)!|S^{2m}|=\Lambda_1.
\end{equation}

\medskip

\noindent\emph{Proof of Proposition \ref{blow}.}  \emph{Step 1.} Set $\sigma_k=e^{-u_k(x_k)}$, and consider on $B_{\frac{\ring}{\sigma_k}}\subset\R{2m}$ the functions
\begin{equation}\label{eq*}
z_k(y):=u_k(\exp_{x_k}(\sigma_k y))+\log(\sigma_k)\leq 0,
\end{equation}
and the metrics
$$\tilde g_k:= (\exp_{x_k}\circ T_k)^* g,$$
where $T_k:\R{2m}\to\R{2m}$, $T_ky=\sigma_k y$.
Then, setting $\hat Q_k(y):=Q_k(\exp_{x_k}(\sigma_k y))$, and pulling back \eqref{presc} via $\exp_{x_k}\circ T_k$, we get
\begin{equation}\label{utilde}
P^{2m}_{\tilde g_k} z_k+Q^{2m}_{\tilde g_k}=\sigma_k^{-2m}\hat Q_ke^{2mz_k}.
\end{equation}
Setting now $\hat g_k:=\sigma_k^{-2} \tilde g_k$, we have $P^{2m}_{\hat g_k}=\sigma_k^{2m}P^{2m}_{\tilde g_k}$, $Q^{2m}_{\hat g_k}=\sigma_k^{2m} Q^{2m}_{\tilde g_k}$, and from \eqref{utilde} we infer
\begin{equation}\label{uhat}
P^{2m}_{\hat g_k} z_k+Q^{2m}_{\hat g_k}=\hat Q_ke^{2mz_k}.
\end{equation}
Then, since the principal part of the Paneitz operator is $(-\Delta_g)^m$, we can write
$$P_{\hat g_k}=(-\Delta_{\hat g_k})^m+A_k,$$
where $A_k$ is a linear differential operator of order at most $2m-1$; moreover the coefficients of $A_k$ are going to $0$ in $C^k_{\loc}(\R{2m})$ for all $k\geq 0$, since $\hat g_k\to g_{\R{2m}}$ in $C^k_{\loc}(\R{2m})$ for all $k\geq 0$, and $P_{g_{\R{2m}}}=(-\Delta)^m$. Then \eqref{uhat} can be written as
\begin{equation}\label{Deltazk}
(-\Delta_{\hat g_k})^m z_k +A_kz_k +Q^{2m}_{\hat g_k}=\hat Q_ke^{2mz_k}.
\end{equation}

\medskip

\noindent \emph{Step 2.} We now claim that $z_k\to z_0$ in $C^{2m-1,\alpha}_{\loc}(\R{2m})$, where
\begin{equation}\label{z0}
(-\Delta)^m z_0 =Q_0(x_0)e^{2m z_0},\quad \int_{\R{2m}} e^{2m z_0}dx<\infty.
\end{equation}
We first assume $m>1$. Fix $R>0$ and write $z_k=h_k+w_k$ on $B_R=B_R(0)$, where $\Delta_{\hat g_k}^m h_k=0$ and 
\begin{equation}\label{wk10}
\left\{
\begin{array}{ll}
(-\Delta_{\hat g_k})^m w_k=(-\Delta_{\hat g_k})^m z_k.&\textrm{in } B_R\\
w_k=\Delta w_k=\ldots=\Delta^{m-1}w_k=0&\textrm{on }\partial B_R
\end{array}
\right.
\end{equation}
From $z_k\leq 0$ we infer $\|\hat Q_ke^{2m z_k}\|_{L^\infty(B_R)}\leq C$, and clearly $Q^{2m}_{\hat g_k}=\sigma_k^{2m} Q^{2m}_{\tilde g_k}\to 0$ in $L^\infty_{\loc}(\R{2m})$.
Lemma \ref{nabla} implies that $(A_kz_k)$ is bounded in $L^p(B_R)$, $1\leq p<\frac{2m}{2m-1}$, hence from \eqref{wk10} and elliptic estimates we get uniform bounds for $(w_k)$ in $W^{2m,p}(B_R)$, $1\leq p<\frac{2m}{2m-1}$, hence in $C^0(B_R)$.
Again using Lemma \ref{nabla}, we get
$$\|\Delta_{\hat g_k} h_k\|_{L^1(B_R)}\leq C( \|z_k\|_{W^{2,1}(B_R)}+\| w_k\|_{W^{2,1}(B_R)})\leq C.$$
Since $\Delta_{\hat g_k}^{m-1}(\Delta_{\hat g_k} h_k)=0,$ elliptic estimates (compare Proposition \ref{c2m}) give 
\begin{equation}\label{hk10}
\|\Delta_{\hat g_k} h_k\|_{C^\ell(B_{R/2})}\leq C(\ell)\quad \textrm{for every }\ell\in \mathbb{N}.
\end{equation}
This, together with $|h_k(0)|=|w_k(0)|\leq C$, and $h_k\leq -w_k\leq C$ and elliptic estimates (e.g. \cite[Thm. 8.18]{GT}), implies that
$\|h_k\|_{L^1(B_{R/2})}\leq C,$ hence, again using elliptic estimates,  
\begin{equation}\label{hk3'}
\|h_k\|_{C^\ell(B_{R/4})}\leq C(\ell) \quad\textrm{for every }\ell\in\mathbb{N}.
\end{equation}
Therefore $(z_k)$ is bounded in $W^{2m,p}(B_{R/4})$, $1\leq p<\frac{2m}{2m-1}$. 
We now go back to \eqref{wk10}, replacing $R$ with $R/4$ and redefining $h_k$ and $w_k$ accordingly on $B_{R/4}$.
We now have that $(A_kz_k)$ is bounded in $L^p(B_{R/4})$ for $1\leq p<\frac{2m}{2m-2}$ by Sobolev's embedding, and we infer as above that $(w_k)$ is bounded in $W^{2m,p}(B_{R/4})$, $1\leq p<\frac{2m}{2m-2}$, and $h_k$ is bounded in $C^\ell(B_{R/16})$, $\ell\geq 0$. Iterating, we find that $(z_k)$ is bounded in $W^{2m,p}(B_{R/4^{2m}})$ for every $p\in [1,\infty[$. By letting $R\to \infty$ and extracting a diagonal subsequence, we infer that $(z_k)$ converges in $C^{2m-1,\alpha}_{\loc}(\R{2m})$. Then \eqref{z0} follows from Fatou's lemma, letting $R\to \infty$, and the claim is proven.

When $m=1$, since $P^2_g=-\Delta_g$, \eqref{uhat} implies at once that $(\Delta_{\hat g_k} z_k)$ is locally bounded in $L^\infty$. Then, since $z_k\leq 0$ and $z_k(0)=0$, the claim follows from elliptic estimates (e.g. \cite[Thm. 8.18]{GT}).

\medskip

\noindent \emph{Step 3.} We shall now rule out the possibility that $Q_0(x_0)\leq 0$.

\medskip

\noindent \emph{Case }$Q_0(x_0)=0$. By the maximum principle one sees that, for $m=1$, \eqref{z0} has no solution (see e.g. \cite[Thm. 3]{mar2}), contradiction. If $m\geq 2$, still by \cite[Thm. 3]{mar2}, any solution $z_0$ to \eqref{z0} is a non-constant polynomial of degree at most $2m-2$, and there are $1\leq j\leq m-1$ and $a<0$ such that $\Delta^j z_0\equiv a$. Following an argument of \cite{RS}, see also \cite{mal}, we shall find a contradiction. Indeed we have
$$\lim_{k\to\infty}\int_{B_R}|\Delta^j z_k|dx = \int_{B_R}|\Delta^j z_0|dx=\frac{|a|\omega_{2m}}{2m}R^{2m}+o(R^{2m}),\quad \textrm{as } R\to+\infty.$$
Scaling back to $u_k$, we find
$$\lim_{k\to\infty}\bigg(\sigma_k^{2j-2m}\int_{B_{R\sigma_k}(x_k)}|\nabla^{2j} u_k|\dvol_g\bigg)\geq C^{-1}R^{2m}+o(R^{2m}),\quad \textrm{as } R\to+\infty,$$
while, from Lemma \ref{nabla},
\begin{equation}\label{deltac}
\int_{B_{R\sigma_k}(x_k)}|\nabla^{2j} u_k| \dvol_g \leq C(R\sigma_k)^{2m-2j}.
\end{equation}
This yields the desired contradiction as $k,R\to+\infty$.

\medskip

\noindent \emph{Case }$Q_0(x_0)<0$. By \cite[Thm. 1]{mar2} there exists no solution to \eqref{z0} for $m=1$, a contradiction. If $m\geq 2$, from \cite[Thm. 2]{mar2} we infer that there are a constant $a\neq 0$ and $1\leq j\leq m-1$ such that
$$\lim_{\substack{|x|\to+\infty\\ x\in \M{C}}}\Delta^j z_0(x)=a,$$
where $\M{C}:=\{t\xi\in \R{2m}:t\geq 0, \; \xi\in K\}$ and $K\subset S^{2m-1}$ is a compact set with $\M{H}^{2m-1}(K)>0$. Then, as above,
\begin{eqnarray*}
\lim_{k\to\infty}\bigg(\sigma_k^{2j-2m}\int_{B_{R\sigma_k}(x_k)}|\nabla^{2j} u_k|\dvol_g\bigg)&\geq&C^{-1} \int_{B_R\cap \M{C}}|\Delta^j z_0|dx\\
&\geq& C^{-1} R^{2m}+o(R^{2m}),
\end{eqnarray*}
again contradicting \eqref{deltac}. Then we have shown that $Q_0(x_0)>0$.

\medskip

\noindent \emph{Step 4.} Since $Q_k(x_0)>0$, $\mu_k$ and $\eta_k$ are well-defined. Repeating the procedure of Step 2, we find a function $\overline \eta\in C^{2m-1,\alpha}_{\loc}(\R{2m})$ such that $\eta_k\to \overline\eta $ in $C^{2m-1,\alpha}_{\loc}(\R{2m}),$
where (compare \eqref{z0})
$$(-\Delta)^m \overline\eta=(2m-1)!e^{2m\overline\eta},\quad \int_{\R{2m}}e^{2m\overline\eta}dx<+\infty.$$
By \cite[Thm. 2]{mar1}, either $\overline\eta$ is a standard solution, i.e. there are $x_0\in\R{2m}$, $\lambda>0$ such that
\begin{equation}\label{v0lambda}
\overline\eta(y)=\log\frac{2\lambda}{1+\lambda^2|y-y_0|^2},
\end{equation}
or $\Delta^j \overline\eta(x)\to a$ as $|x|\to\infty$ for some constant $a<0$ and for some $1\leq j\leq m-1$. In the latter case, as in Step 3, we reach a contradiction. Hence \eqref{v0lambda} is satisfied.
Since $\max_M \eta_k=\eta_k(0)=\log 2$ for every $k$, we have $y_0=0$, $\lambda=1$, i.e. $\overline \eta=\eta_0$.
Since, by Fatou's lemma
$$\lim_{R\to\infty}\lim_{k\to\infty}\int_{R\mu_k(x_k)}Q_k e^{2mu_k}\dvol_g =(2m-1)!\int_{\R{2m}}e^{2m\eta_0}dx,$$
\eqref{energia} follows from \eqref{etak2}.
\hfill$\square$

\medskip

\noindent\emph{Proof of Theorem \ref{trmquant}.} Assume first that $u_k\leq C$. Then $P_g^{2m} u_k$ is bounded in $L^\infty(M)$ and Lemma \ref{nabla} and by elliptic estimates $u_k-\overline u_k$ is bounded in $W^{2m,p}(M)$ for every $1\leq p<\infty,$ hence in $C^{2m-1,\alpha}(M)$ for every $\alpha\in [0,1[$, where $\overline u_k:=\intm_M u_k \dvol_g$. Observe that by Jensen's inequality and \eqref{volume}, $\overline u_k\leq C$.

If $\overline u_k$ remains bounded (up to a subsequence), then by Ascoli-Arzel\`a's theorem, for every $\alpha\in [0,1[$, $u_k$ is convergent (up to a subsequence) in $C^{2m-1,\alpha}(M)$, and we are in case (i) of Theorem \ref{trmquant}. 

If $\overline u_k\to-\infty$, we have that $u_k\to-\infty$ uniformly on $M$ and we are in case (ii) of the theorem, with $S_1=\emptyset$.

From now on we shall assume that $\max_M u_k\to \infty$ as $k\to \infty$, and closely follow the argument of \cite{DR}.

\medskip

\noindent\emph{Step 1.} There are $I>0$ converging sequences $x_{i,k}\to x^{(i)}\in M$ with $u_k(x_{i,k})\to\infty$ as $k\to\infty$, such that
\begin{enumerate}

\item[($A_1$)] $Q_0(x^{(i)})>0$, $1\leq i\leq I$.

\item[($A_2$)] $\frac{\dist(x_{i,k},x_{j,k})}{\mu_{i,k}}\to+\infty$ as $k\to +\infty$ for all $1\leq i,j\leq I$, $i\neq j$, where
$$\mu_{i,k}:=2\bigg(\frac{(2m-1)!}{Q_0(x^{(i)})}\bigg)^\frac{1}{2m}e^{-u_k(x_{i,k})}.$$

\item[($A_3$)] Set $\eta_{i,k}(y):=u_k(\exp_{x_{i,k}}(\mu_{i,k}y))-u_k(x_{i,k})$. Then for $1\leq i\leq I$
\begin{equation}\label{vkv0}
\eta_{i,k}(y)\to \eta_0(y)=\log\frac{2}{1+|y|^2} \quad \textrm{in } C^{2m}_{\loc}(\R{2m})\quad (k\to\infty).
\end{equation}

\item[($A_4$)] For $1\leq i\leq I$
\begin{equation}\label{eqA4}
\lim_{R\to+\infty}\lim_{k\to+\infty}\int_{B_{R\mu_{i,k}}(x_{i,k})}Q_k e^{2mu_k}dx\to \Lambda_1.
\end{equation}

\item[($A_5$)] There exists $C>0$ such that for all $k$
$$\sup_{x\in M}\big[e^{u_k(x)}R_k(x)\big]\leq C, \quad R_k(x):=\min_{1\leq i\leq I}\dist(x,x_{i,k}).$$
\end{enumerate}
Step 1 follows from Proposition \ref{blow} and induction as follows. Define $x_{1,k}=x_k$ as in Proposition \ref{blow}. Then ($A_1$), ($A_3$) and ($A_4$) are satisfied with $i=1$. If
$\sup_{x\in M}\big[e^{u_k(x)}\dist(x_{1,k},x)\big]\leq C,$
then $I=1$ and also ($A_5$) is satisfied, so we are done. Otherwise we choose $x_{2,k}$ such that
\begin{equation}\label{r1k}
R_{1,k}(x_{2,k})e^{u_k(x_{2,k})}=\max_{x\in M} R_{1,k}(x)e^{u_k(x)}\to\infty, \quad R_{1,k}(x):=\dist (x,x_{1,k}).
\end{equation}
Then ($A_2$) with $i=2$, $j=1$ follows at once from \eqref{r1k}, while ($A_2$) with $i=1$, $j=2$ follows from ($A_3$), as in \cite{DR}.
A slight modification of Proposition \ref{blow} shows that $(x_{2,k}, \mu_{2,k})$ satisfies ($A_1$), ($A_3$) and ($A_4$), and we continue so, until also property ($A_5$) is satisfied. The procedure stops after finitely many steps, thanks to ($A_2$), ($A_4$) and \eqref{volume}.

\medskip

\noindent\emph{Step 2.} With the same proof as in Step 2 of \cite[Thm. 1]{DR}:
\begin{equation}\label{ddist}
\sup_{x\in M} R_k(x)^\ell|\nabla^\ell u_k(x)|\leq C,\quad \ell=1,2,\ldots,2m-1.
\end{equation}

\medskip

\noindent\emph{Step 3.} $u_k\to-\infty$ locally uniformly in $M\backslash S_1$, $S_1:=\{x^{(i)}:1\leq i\leq I\}$. This follows easily from \eqref{ddist} above and \eqref{nu} below (which implies that $u_k\to-\infty$ locally uniformly in $B_{\delta_\nu}(x^{(i)})\backslash \{x^{(i)}\}$ for any $1\leq i\leq I$, $\nu\in [1,2[$ and $\delta_\nu$ as in Step 4), but we also sketch an instructive alternative proof, which does not make use of \eqref{nu}.

Our Theorem \ref{main} can be reproduced on a closed manifold, with a similar proof and using Proposition 3.1 from \cite{mal} instead of Theorem \ref{a2m} above. Then either

\begin{itemize}
\item[(a)] $u_k$ is bounded in $C^{2m-1,\alpha}_{\loc}(M\backslash S_1)$, or
\item[(b)] $u_k\to -\infty$ locally uniformly in $M\backslash S_1$, or
\item[(c)] There exists a closed set $S_0\subset M\backslash S_1$ of Hausdorff dimension at most $2m-1$ and numbers $\beta_k\to+\infty$ such that
\begin{equation}\label{varphi1}
\frac{u_k}{\beta_k}\to\varphi \textrm{ in } C^{2m-1,\alpha}_{\loc}(M\backslash (S_0 \cup S)),
\end{equation}
where
\begin{equation}\label{varphi2}
\Delta_g^m\varphi \equiv 0,\quad \varphi\leq 0,\quad \varphi\not\equiv 0\quad \text{on }M\backslash S_1,\quad \varphi\equiv 0\textrm{ on } S_0.
\end{equation}
\end{itemize} 
Case (a) can be ruled out using \eqref{volume} as in \eqref{fatou} at the end of the proof of Theorem \ref{main}. Case (c) contradicts Lemma \ref{nabla}, by considering any ball $B_R(x_0)\subset\subset \Omega\backslash S_1$ with $\int_{B_R(x_0)}|\nabla\varphi|\dvol_g>0$ and using \eqref{varphi1}.
Hence Case (b) occurs, as claimed.

\medskip

\noindent\emph{Step 4.}
We claim that for every $1\leq \nu<2$, there exist $\delta_\nu>0$ and $C_\nu>0$ such that for $1\leq i\leq I$
\begin{equation}\label{nu}
\dist(x,x_{i,k})^{2m\nu}e^{2mu_k(x)}\leq C_\nu\mu_{i,k}^{2m(\nu-1)},\quad \textrm{for } x\in B_{\delta_\nu}(x_{i,k}).
\end{equation}
Then on the \emph{necks} $\Sigma_{i,k}:=B_{\delta_\nu}(x_{i,k})\backslash B_{R\mu_{i,k}}(x_{i,k})$ we have
\begin{eqnarray*}
\int_{\Sigma_{i,k}}e^{2mu_k}\dvol_g&\leq& C_\nu\mu_{i,k}^{2m(\nu-1)}\int_{\Sigma_{i,k}}\dist(x,x_{i,k})^{-2m\nu}\dvol_g(x)\\
&\leq &C_\nu\mu_{i,k}^{2m(\nu-1)}\int_{R\mu_{i,k}}^{\delta_\nu}r^{2m-1-2m\nu}dr\\
&=&C_\nu R^{2m(1-\nu)}-C_\nu \mu_{i,k}^{2m(\nu-1)}\delta_\nu^{2m(1-\nu)},
\end{eqnarray*}
whence
\begin{equation}\label{vanish}
\lim_{R\to+\infty}\lim_{k\to +\infty}\int_{\Sigma_{i,k}}Q_ke^{2mu_k}\dvol_g =0.
\end{equation}
This, together with \eqref{energia} and Step 3 implies \eqref{QgL}, assuming that $x^{(i)}\neq x^{(j)}$ for $i\neq j$. This we be shown in Step 4c below. Then \eqref{intQg} follows at once from \eqref{total}.

\medskip

Let us prove \eqref{nu}.
Fix $1\leq \nu<2$ and set for $1\leq i\leq I$
$$\tilde R_{i,k}:=\min_{j\neq i} \dist(x_{i,k}, x_{j,k}).$$

\medskip

\noindent\emph{Step 4a.} Let $i\in\{1,\ldots, I\}$ be such that for some $\theta>0$ we have
\begin{equation}\label{theta}
\tilde R_{i,k}\leq \theta \tilde R_{j,k}\quad \textrm{for } 1\leq j\leq I,\; k\geq 1.
\end{equation}
Set
\begin{equation}\label{46'}
\varphi_{i,k}(r):=r^{2m\nu}\exp\bigg( \Intm_{\partial B_r(x_{i,k})}2mu_kd\sigma_g\bigg),
\end{equation}
for $0<r<\ring$, where $d\sigma_g$ is the measure on $\partial B_r(x_{i,k})$ induced by $g$.
Observe that
\begin{equation}\label{phi'}
\varphi_{i,k}'(r\mu_{i,k})<0 \quad \textrm{ if and only if }\quad
r\mu_{i,k}< -\nu\bigg(\Intm_{\partial B_{r\mu_{i,k}}(x_{i,k})} \frac{\partial u_k}{\partial n}d\sigma_g\bigg)^{-1}.
\end{equation}
From \eqref{vkv0} we infer
$$\mu_{i,k}\frac{\partial u_k}{\partial n}\bigg|_{\partial B_{\mu_{i,k} r}(x_{i,k})}\to\frac{\partial }{\partial r}\log\frac{2}{1+r^2}=\frac{-2 r}{1+r^2},$$
hence
$$\mu_{i,k}\Intm_{\partial B_{\mu_{i,k} r}(x_{i,k})}\frac{\partial u_k}{\partial n}d\sigma_g\to -\frac{2r}{1+r^2},\quad \textrm{for }r>0 \textrm{ as }k\to\infty,$$
and \eqref{phi'} implies that for any $R\geq 2R_\nu:=2\sqrt{\frac{\nu}{2-\nu}}$,  there exists $k_0(R)$ such that 
\begin{equation}\label{k_0}
\varphi_{i,k}'(r\mu_{i,k})<0\quad \textrm{for }k\geq k_0(R),\;r\in[2R_\nu,R].
\end{equation}
Define
\begin{equation}\label{defrk}
r_{i,k}:=\sup\Big\{r\in[2R_\nu\mu_{i,k},\tilde R_{i,k}/2]: \varphi_{i,k}'(\rho)<0 \textrm{ for }\rho\in[2R_\nu\mu_{i,k},r) \Big\}.
\end{equation}
From \eqref{k_0} we infer that
\begin{equation}\label{rmu}
\lim_{k\to+\infty}\frac{r_{i,k}}{\mu_{i,k}}=+\infty.
\end{equation}
Let us assume that
\begin{equation}\label{D24}
\lim_{k\to\infty}r_{i,k}=0.
\end{equation}
Consider
\begin{equation}\label{**}
v_{i,k}(y):=u_k(\exp_{x_{i,k}}(r_{i,k} y))-C_{i,k},\quad C_{i,k}:=\Intm_{\partial B_{r_{i,k}}(x_{i,k})}u_kd\sigma_g,
\end{equation}
and let
$$\hat g_{i,k}:=r_{i,k}^{-2}(\exp_{x_{i,k}}\circ T_{i,k})^*g,\quad \hat Q_{i,k}(y):= Q_k(\exp_{x_{i,k}}(r_{i,k} y)),$$
where
$$T_{i,k}(y):= r_{i,k}y\quad\textrm{for } y\in \R{2m}.$$
Then
\begin{eqnarray}
P^{2m}_{\hat g_{i,k}} v_{i,k} + r_{i,k}^{2m} Q_{\hat g_{i,k}}&=&r_{i,k}^{2m}\hat Q_{i,k}e^{2m (v_{i,k}+C_{i,k})} \nonumber\\
&=& r_{i,k}^{2m(1-\nu)}\varphi_{i,k}(r_{i,k})\hat Q_{i,k}e^{2m v_{i,k}}.\label{deltavk}
\end{eqnarray}
We also set
\begin{equation}\label{Ji}
\mathcal{J}_i=\{j\neq i: \dist(x_{i,k},x_{j,k})=O(r_{i,k})\textrm{ as } k\to\infty\},
\end{equation}
and
\begin{equation}\label{D27}
\tilde x_{j,k}^{(i)}:=\frac{1}{r_{i,k}}\exp^{-1}_{x_{i,k}}(x_{j,k}),\quad \tilde x_j^{(i)}=\lim_{k\to\infty} \tilde x_{j,k},
\end{equation}
after passing to a subsequence, if necessary. Thanks to \eqref{theta} and \eqref{defrk}, we have that $|\tilde x_j^{(i)}|\geq 2$ for all $j\in \mathcal{J}_i$ and that
$$|\tilde x_j^{(i)}-\tilde x_\ell^{(i)}|\geq \frac{2}{\theta}\quad \textrm{for all } j,\ell\in\mathcal{J}_i, \;j\neq \ell.$$
By \eqref{ddist} and the choice of $C_{i,k}$ in \eqref{**}, $v_{i,k}$ is uniformly bounded in
$$C^{2m-1}_{\loc}(\R{2m}\backslash \{0, \tilde x_j^{(i)}:j\in \mathcal{J}_i\}).$$
Thanks to \eqref{defrk} and \eqref{rmu}, given $R> 2R_\nu$, there exists $k_0(R)$ such that
$\varphi_{i,k}(r_{i,k})< \varphi_{i,k}(R\mu_{i,k})$ for all $k\geq k_0$.
From \eqref{vkv0}, we infer
\begin{eqnarray}
\mu_{i,k}^{2m}\exp\bigg(\Intm_{\partial B_{R\mu_{i,k}(x_{i,k})}} 2m u_kd\sigma\bigg)&=&\exp\bigg(\Intm_{\partial B_{R\mu_{i,k}}(x_{i,k})} 2m(u_k+\log \mu_{i,k}) d\sigma\bigg)\nonumber\\
&=& C(R)+o(1),\quad \textrm{as } k\to\infty,\label{CR}
\end{eqnarray}
where
\begin{equation}\label{CR2}
C(R)\to 0, \quad \text{as } R\to\infty.
\end{equation}
Then, together with \eqref{rmu}, letting $k\to+\infty$ we get
\begin{eqnarray}
r_{i,k}^{2m(1-\nu)}\varphi_{i,k}(r_{i,k})&\leq &r_{i,k}^{2m(1-\nu)}\varphi_{i,k}(R\mu_{i,k})\nonumber\\
&=&\mu_{i,k}^{2m}\exp\bigg(\Intm_{\partial B_{R\mu_{i,k}(x_{i,k})}} 2m u_kd\sigma\bigg) R^{2m\nu}\bigg(\frac{\mu_{i,k}}{r_{i,k}}\bigg)^{2m(\nu-1)}\nonumber\\
&\to& 0. \label{rphi}
\end{eqnarray}
Therefore the right-hand side of \eqref{deltavk} goes to $0$ locally uniformly in
$$\R{2m}\backslash \{0,\tilde x_j^{(i)}:j\in \mathcal{J}_i\};$$
moreover
\begin{equation}\label{gk0}
\hat g_{i,k}\to g_{\R{2m}} \text{ in } C^k_{\loc}(\R{2m}) \text{ for every } k\geq 0, \quad r_{i,k}^{2m}\hat Q_{i,k}\to 0\text{ in }C^1_{\loc}(\R{2m}).
\end{equation}
It follows that, up to a subsequence,
\begin{equation}\label{vkh}
v_{i,k}\to h_i \textrm{ in } C^{2m-1,\alpha}_{\loc}(\R{2m}\backslash\{0,\tilde x_j^{(i)}:j\in\mathcal{J}_i\}),
\end{equation}
where, taking \eqref{ddist} into account,
$$\Delta^m h_i(x)= 0,\quad x\in\R{2m}\backslash\{0,\tilde x_j^{(i)}:j\in\mathcal{J}_i\}$$
and
$$\tilde R(x)^\ell|\nabla^\ell h_i(x)|\leq C_\ell, \quad \textrm{for }\ell=1,\ldots, 2m-1,\; x\in\R{2m}\backslash\{0,\tilde x_j^{(i)}:j\in\mathcal{J}_i\},$$
with $\tilde R(x):=\min\{|x|,|x-\tilde x_j^{(i)}|:j\in\mathcal{J}_i\}$.
Then Proposition \ref{propliou} from the appendix implies that
\begin{equation}\label{hi}
h_i(x)=-\lambda\log |x|-\sum_{j\in\mathcal{J}_i}\lambda_j \log|x-\tilde x_j^{(i)}|+\beta,
\end{equation}
for some $\lambda,\beta,\lambda_j\in\R{}$. We now recall that the Paneitz operator is in divergence form, hence we can write
\begin{equation}\label{pan}
P^{2m}_{\hat g_{i,k}} v_{i,k}=\diver_{\hat g_{i,k}}(A_{\hat g_{i,k}} v_{i,k})
\end{equation}
for some differential operator $A_{\hat g_{i,k}}$ of order $2m-1$, with coefficients converging to the coefficient of $(-1)^m\nabla\Delta^{m-1}$ uniformly in $B_1$, thanks to \eqref{gk0}. Then integrating \eqref{deltavk}, using \eqref{gk0}, \eqref{vkh} and \eqref{pan}, we get
\begin{eqnarray}
\lim_{k\to\infty}\int_{B_{r_{i,k}}(x_{i,k})}Q_k e^{2mu_k}\dvol_g
&=&\lim_{k\to\infty}\varphi_{i,k} (r_{i,k})r_{i,k}^{2m(1-\nu)}\int_{B_1}\hat Q_{i,k} e^{2mv_{i,k}}\dvol_{\hat g_{i,k}}\nonumber\\
&=&\lim_{k\to\infty}\int_{B_1}\Big(\diver_{\hat g_{i,k}} (A_{\hat g_{i,k}}v_{i,k})+r_{i,k}^{2m} Q_{\hat g_{i,k}}\Big)\dvol_{\hat g_{i,k}}\nonumber\\
&=&\lim_{k\to\infty}\int_{\partial B_1}n\cdot (A_{\hat g_{i,k}}v_{i,k}) d\sigma_{\hat g_{i,k}}\nonumber\\
&=&(-1)^m\int_{\partial B_1}\frac{\partial \Delta^{m-1}h_i}{\partial n}d\sigma =\lambda\frac{\Lambda_1}{2} \label{limuk} ,
\end{eqnarray}
where here $n$ denotes the exterior unit normal to $\de B_1$ and the last identity can be inferred using \eqref{fund} and the following:
\begin{eqnarray*}
\int_{\partial B_1}\frac{\partial \Delta^{m-1}h_i}{\partial n}d\sigma &=&\lambda\int_{\partial B_1}\frac{\partial \Delta^{m-1}\log\frac{1}{|x|}}{\partial n}d\sigma\\
&&+\sum_{j\in \M{J}_i}\lambda_j\int_{B_1}\underbrace{\Delta^m \log\frac{1}{|x-\tilde x_j^{(i)}|}}_{\equiv 0 \; on \; B_1}dx  
\end{eqnarray*}
From \eqref{ddist} with $\ell=1$, we get
\begin{equation}\label{uk}
|u_k(\exp_{x_{i,k}}(r_{i,k} y_1))-u_k(\exp_{x_{i,k}}( r_{i,k} y_2))|\leq Cr_{i,k}r\sup_{\partial B_{r_{i,k}r}(x_{i,k})}|\nabla u_k|\leq C,
\end{equation}
for $0\leq r\leq \frac{3}{2}$, $|y_1|=|y_2|=r$.
For $2R_\nu\mu_{i,k}\leq R\mu_{i,k}\leq r\leq r_{i,k}$, we infer from \eqref{CR}
$$
\varphi_{i,k}(r)\leq \varphi_{i,k}(R \mu_{i,k})
\leq C(R) \mu_{i,k}^{2m(\nu-1)}+o(\mu_{i,k}^{2m(\nu-1)}) .
$$
This, \eqref{46'}, \eqref{CR}, \eqref{CR2} and \eqref{uk} imply that for any $\eta>0$ there exist $R_\eta\geq 2R_\nu$ and $k_\eta\in\mathbb{N}$ such that
\begin{equation}\label{eq47}
\dist(x,x_{i,k})^{2m\nu}e^{2mu_k}\leq\eta \mu_{i,k}^{2m(\nu-1)}\quad \textrm{for } x\in B_{r_{i,k}}(x_{i,k})\backslash B_{R_\eta \mu_{i,k}}(x_{i,k}),\; k\geq k_\eta.
\end{equation}
It now follows easily that
$$\lim_{R\to+\infty}\lim_{k\to\infty}\int_{B_{r_{i,k}}(x_{i,k})\backslash B_{R\mu_{i,k}}(x_{i,k})}Q_ke^{2mu_k}dx=0,$$
and from \eqref{eqA4}
$$\lim_{k\to+\infty}\int_{B_{r_{i,k}}(x_{i,k})}Q_ke^{2mu_k}dx=\Lambda_1.$$
That implies that $\lambda=2$.
With a similar computation, integrating on $B_\delta(\tilde x_j^{(i)})$ for $\delta$ small instead of $B_1(0)$, one proves that $\lambda_j\geq 2$ for all $j\in\mathcal{J}_i$. Now set
$$\overline h_i(r):=\Intm_{\partial B_r(0)}h_id\sigma. $$
Then
$$\frac{d}{dr}\big(r^{2m\nu}e^{2m\overline h_i(r)}\big)=2m\bigg(\nu-2-\bigg(\sum_{j\in\mathcal{J}_i}\frac{\lambda_j}{2|\tilde x_j^{(i)}|^2}\bigg)r^2\bigg)r^{2m\nu-1}e^{2m\overline h_i(r)},$$
for $0<r<\frac{3}{2}$.
In particular
$$\frac{d}{dr}\big(r^{2m\nu}e^{2m\overline h_i(r)}\big)\big|_{r=1}<0$$
hence, for $k$ large enough, $\varphi_{i,k}'(r_{i,k})<0$. This implies that
\begin{equation}\label{D36}
r_{i,k}=\frac{\tilde R_{i,k}}{2} \quad \textrm{for $k$ large.}
\end{equation}
This in turn implies $\lim_{k\to\infty}\tilde R_{i,k}=0$, when $i$ satisfies \eqref{theta} and $\lim_{k\to \infty}r_{i,k}=0$. For $i$ satisfying \eqref{theta} and $\limsup_{k\to\infty} \tilde R_{i,k}>0$, we infer, instead, that $\limsup_{k\to \infty}r_{i,k}>0$. In both cases \eqref{eq47} holds. 

\medskip
 
\noindent\emph{Step 4b.} Now assume that
\begin{equation}\label{limR}
\limsup_{k\to\infty} \tilde R_{i,k}> 0,\quad \text{for every } 1\leq i\leq I.
\end{equation}
Then \eqref{theta} is satisfied for every $1\leq i\leq I$, hence $\limsup_{k\to\infty}r_{i,k}>0$, $1\leq i\leq I$. Up to selecting a subsequence, we can set
$$\delta_\nu:=\inf_{1\leq i\leq I}\frac{1}{2}\lim_{k\to\infty} r_{i,k}>0.$$
Take now $\eta=1$ in \eqref{eq47}, and let $R_1$ be the corresponding $R_\eta$. Then \eqref{nu} is true for $x\in B_{\delta_\nu}(x_{i,k})\backslash B_{R_1\mu_{i,k}}(x_{i,k})$. On the other hand, thanks to ($A_3$), we have
$u_k(x)\leq u_k(x_{i,k})+ C$ on $B_{R_1\mu_{i,k}}(x)$. Then, using \eqref{defmuk}, we get
\begin{eqnarray*}
\dist(x,x_{i,k})^{2m\nu} e^{2mu_k(x)}&\leq& C(R_1\mu_{i,k})^{2m\nu}e^{2mu_k(x_{i,k})}\\
&\leq& CR_1^{2m\nu}\mu_{i,k}^{2m(\nu-1)} \quad \textrm{for }  x\in B_{R_1\mu_{i,k}}(x_{i,k}).
\end{eqnarray*}
This completes the proof of \eqref{nu}, under the assumption that \eqref{limR} holds.

\medskip

\noindent\emph{Step 4c.} We now prove that in fact \eqref{limR} holds true. Choose $1\leq i_0\leq I$ so that, up to a subsequence,
$$\tilde R_{i_0,k}=\min_{1\leq i\leq I}\tilde R_{i,k}\quad \text{for every } k\in\mathbb{N},$$
and assume by contradiction that $\lim_{k\to \infty}\tilde R_{i_0,k}=0.$
Clearly \eqref{theta} holds for $i=i_0$, hence also \eqref{D36} holds for $i=i_0$, by Step 4a. Then, setting $\mathcal{J}_{i_0}$ as is \eqref{Ji},  we claim that, for any $i\in \mathcal J_{i_0}$, there exists $\theta(i)>0$ such that
$$\tilde R_{i,k}\leq \theta(i)\tilde R_{j,k}\quad \textrm{for }1\leq j\leq I.$$
Indeed
$$\tilde R_{i,k}=O(r_{i_0,k})=O(\tilde R_{i_0,k}) \quad\textrm{as } k\to\infty.$$
It then follows that \eqref{theta} holds for all $i\in \mathcal{J}_{i_0}$, and that Step 4a applies to them. Observing that $\mathcal{J}_{i_0}\neq \emptyset$ thanks to Step 4a (Identity \eqref{D36} with $i_0$ instead of $i$), we can pick $i\in \M{J}_{i_0}$ such that, up to a subsequence,
$$\dist(x_{i,k}, x_{i_0,k})\geq \dist (x_{j,k}, x_{i_0,k})\quad \textrm{for all }j\in \M{J}_{i_0},\;k>0.$$
Recalling the definition of $\tilde x_j^{(i)}$ for $j\in \M{J}_i$, we get $|\tilde x_{i_0}^{(i)}|\geq|\tilde x_j^{(i)}-\tilde x_{i_0}^{(i)}|$ for all $j\in \M{J}_{i}$.
A consequence of this inequality is that the scalar product
\begin{equation}\label{D37}
\tilde x_{i_0}^{(i)}\cdot\tilde x_j^{(i)}>0
\end{equation}
for all $j\in\M{J}_i$. In other words all the $\tilde x_j^{(i)}$'s with $j\in \M{J}_i$ lie in the same half space orthogonal to $\tilde x_{i_0}^{(i)}$ and whose boundary contains $0=\tilde x_i^{(i)}$. Multiplying \eqref{deltavk} by $\nabla v_{i,k}$ and integrating over $B_\delta=B_\delta(0)$ ($\delta>0$ small), we get
\begin{eqnarray}
\int_{B_\delta}P^{2m}_{\hat g_{i,k}}v_{i,k}\nabla v_{i,k}\dvol_{\hat g_{i,k}}&=&-\int_{B_\delta}r_{i,k}^{2m}\hat Q_{i,k}\nabla v_{i,k}\dvol_{\hat g_{i,k}}\nonumber \\
&&+\frac{r_{i,k}^{2m(1-\nu)}}{2m}\varphi_{i,k}(r_{i,k})\int_{B_\delta(0)}\hat Q_{i,k}\nabla e^{2m v_{i,k}}\dvol_{\hat g_{i,k}}\nonumber\\
&=:&(I)_k+(II)_k.\label{III}
\end{eqnarray}
Recalling \eqref{gk0} and \eqref{vkh}, we see at once that $\lim_{k\to\infty} (I)_k=0.$
Integrating by parts, we also see that
\begin{eqnarray*}
|(II)_k|&\leq & C\frac{r_{i,k}^{2m(1-\nu)}}{2m}\varphi_{i,k}(r_{i,k})\int_{B_\delta(0)}\frac{\nabla\hat Q_{i,k}}{\hat Q_{i,k} }\hat Q_{i,k}e^{2m v_{i,k}}d\vol_{\hat g_{i,k}}\\
&&+\frac{r_{i,k}^{2m(1-\nu)}}{2m}\varphi_{i,k}(r_{i,k})\int_{\de B_\delta(0)} O(1)d\sigma_{\hat g_{i,k}}\\
&\to& 0\quad \text{as }k\to \infty,
\end{eqnarray*}
where the last term vanishes thanks to \eqref{rphi}, and the first term on the right of $(II)_k$ vanishes thanks to \eqref{limuk} and the
remark that
\begin{equation}\label{nablaQ}
\frac{\nabla\hat Q_{i,k} }{\hat Q_{i,k}}\to 0 \quad \text{in }L^\infty(B_\delta).
\end{equation}
Recalling \eqref{vkh}, using \eqref{ddist} and \eqref{gk0}, we arrive at
\begin{equation}\label{D38}
\int_{B_\delta}\nabla h_i(-\Delta)^m h_i dx=0.
\end{equation}
Let us assume $m$ even. Then, integrating by parts, we get
\begin{eqnarray}
0&=&\frac{1}{2}\int_{\partial B_\delta}((-\Delta)^\frac{m}{2}h_i)^2 n d\sigma\nonumber\\
&&-\sum_{j=0}^{\frac{m}{2}-1}\int_{\partial B_\delta}(\nabla (-\Delta)^j h_i)\frac{\partial (-\Delta)^{m-1-j}h_i}{\de n} d\sigma\label{parti}\\
&&+\sum_{j=0}^{\frac{m}{2}-1}\int_{\partial B_\delta}\nabla \bigg(\frac{\partial(-\Delta)^j h_i}{\de n}\bigg)(-\Delta)^{m-1-j}h_i d\sigma.\nonumber
\end{eqnarray}
Then, taking the limit as $\delta\to 0$, and writing
$$h_i(x)=2\log\frac{1}{|x|}+G_i(x)$$
we see that all terms in \eqref{parti} vanish ($G_i$ is regular in a neighborhood of $0$ and the vector function $\nabla \log\frac{1}{|x|}$ is anti-symmetric), up to at most
$$\lim_{\delta\to 0}\int_{\partial B_\delta}(-\nabla G_i)\partial_\nu(-\Delta)^{m-1}\Big(2\log\frac{1}{|x|}\Big)d\sigma=2\gamma_m\nabla G_i(0),$$
see \eqref{fund}.
But then \eqref{parti} gives
$$2\gamma_m \nabla G_i(0).$$ 
Also when $m$ is odd, in a completely analogous way, we get $\nabla G_i(0)=0$, a contradiction with \eqref{hi} and \eqref{D37}. This ends the proof of Step 4.

\medskip

\noindent \emph{Step 5.} Finally, if case (ii) occurs and $S\neq \emptyset$, then \eqref{eqA4} implies
$$\limsup_{k\to\infty} \mathrm{vol}(g_k)\geq Q_0(x^{(1)})^{-1}\Lambda_1>0.$$
This justifies the last claim of the theorem.
\hfill $\square$

\section{The case $M=S^{2m}$}\label{sectionsphere}

In the case of the $2m$-dimensional sphere, the concentration-compactness of Theorem \ref{trmquant} becomes quite explicit: only one concentration point can appear and, by composing with suitable M\"obius transformations, we have a global understanding of the concentration behavior. This was already noticed in \cite{str} and \cite{MS}, in dimension $2$ and $4$ under the assumption, which we now drop, that the $Q$-curvatures are positive.

\begin{trm}\label{trmquant2} Let $(S^{2m},g)$ be the $2m$-dimensional round sphere, and let $u_k:M\to \R{}$ be a sequence of solutions of 
\begin{equation}\label{deltaT2}
P_{g} u_k+ (2m-1)!=Q_k e^{2m u_k},
\end{equation}
where $Q_k\to Q_0$ in $C^0$ for a given continuous function $Q_0$. Assume also that
\begin{equation}\label{vkT2}
\mathrm{vol}(g_k)=\int_{S^{2m}}e^{2mu_k}\mathrm{dvol}_g=|S^{2m}|,
\end{equation}
where $g_k:=e^{2mu_k}g$. Then one of the following is true.
\begin{enumerate}
\item[(i)] For every $0\leq \alpha<1$, a subsequence converges in $C^{2m-1,\alpha}(S^{2m})$. 
\item[(ii)] There is a point $x_0\in S^{2m}$ such that up to a subsequence $u_k\to -\infty$ locally uniformly in $S^{2m}\bs \{x_0\}$. Moreover $Q_0(x_0)>0$,
$$Q_k e^{2m u_k}\dvol_{g}\rightharpoonup \Lambda_1\delta_{x_0}$$
and there exist M\"obius diffeomorphisms $\Phi_k$ such that the metrics $h_k:=\Phi_k^*g_k$ satisfy
\begin{equation}\label{hk}
h_k\to g\;\textrm{in } H^{2m}(S^{2m}),\quad Q_{h_k}\to (2m-1)!\;\textrm{in }L^2(S^{2m}).
\end{equation}
\end{enumerate}
\end{trm}

\begin{proof} On the round sphere $P_g=\prod_{i=0}^{m-1}(-\Delta_g +i(2m-i-1))$; moreover $\ker \Delta_g=\{constants\}$ and the non-zero eigenvalues of $-\Delta_g$ are all positive. That easily implies that $\ker P^{2m}_g=\{constants\}$. From Theorem \ref{trmquant}, and the Gauss-Bonnet-Chern theorem, we infer that in case (ii) we have
$$\Lambda_1=\int_{M}Q_g\dvol_g=I\Lambda_1,$$
hence $I=1$, and $Q_ke^{2mu_k}\dvol_g\rightharpoonup \Lambda_1\delta_{x_0}.$ In fact, in order to apply Theorem \ref{trmquant}, we would need $Q_k\to Q_0$ in $C^1(M)$, but this hypothesis is only used in \eqref{nablaQ} in the last part of the proof of Theorem \ref{trmquant}, in order to show that the concentration points are isolated. Since in the case of the sphere only one concentration point appears, that part of the proof is superfluous, and the assumption $Q_k\to Q_0$ in $C^0(M)$ suffices.

To prove the second part of the theorem, for every $k$ we define a M\"obius transformation $\Phi_k:S^{2m}\to S^{2m}$ such that the \emph{normalized metric} $h_k:=\Phi_k^* g_k$ satisfies
$$\int_{S^{2m}}x\dvol_{h_k}=0.$$
Then \eqref{hk} follows by reasoning as in \cite[bottom of Page 16]{MS}.
\end{proof}

\section*{Appendix}

\appendix

\section{A few useful results}

Here we collect a few results which have been used above. For the proofs of Lemma \ref{lemmapiz}, Propositions \ref{c2m} and \ref{propmax}, and Theorem \ref{trmliou}, see e.g. \cite{mar1}.

The following Lemma can be considered a generalized mean value identity for polyharmonic function.

\begin{lemma}[Pizzetti \cite{Piz}]\label{lemmapiz} Let $\Delta^m h=0$, in $B_R(x_0)\subset \R{n}$, for some $m,n$ positive integers. Then there are positive constants $c_i=c_i(n)$ such that
\begin{equation}\label{pizzetti}
\Intm_{B_R(x_0)}h(z)dz =\sum_{i=0}^{m-1}c_iR^{2i}\Delta^i h(x_0).
\end{equation}
\end{lemma}

\begin{prop}\label{c2m} Let $\Delta^{m}h=0$ in $B_{2}\subset\R{n}$. For every $0\leq \alpha<1$, $p\in [1,\infty)$ and $\ell\geq 0$ there are  constants $C(\ell,p)$ and $C(\ell,\alpha)$ independent of $h$ such that
\begin{eqnarray*}
\|h\|_{W^{\ell,p}(B_1)}&\leq &C(\ell,p) \|h\|_{L^1(B_2)}\\
\|h\|_{C^{\ell,\alpha}(B_1)}&\leq& C(\ell,\alpha)  \|h\|_{L^1(B_2)}.
\end{eqnarray*}
\end{prop}

A simple consequence of Lemma \ref{lemmapiz} and Proposition \ref{c2m} is the following Liouville-type Theorem.

\begin{trm}\label{trmliou} Consider $h:\R{n}\to \R{}$ with $\Delta^m h=0$ and $h(x)\leq C(1+|x|^\ell)$ for some integer $\ell\geq 0$. Then $h$ is a polynomial of degree at most $\max\{\ell,2m-2\}$. 
\end{trm}

\begin{prop}\label{propmax} Let $u\in C^{2m-1}(\overline{B}_1)$ such that
\begin{equation}\label{maxpr}
\left\{
\begin{array}{ll}
(-\Delta)^m u\leq C &\textrm{in } B_1\\
(-\Delta)^j u\leq C & \textrm{on }\partial B_1 \textrm{ for }0\leq j<m.
\end{array}
\right.
\end{equation}
Then there exists a constant $C$ independent of $u$ such that
$u\leq C$ in $B_1.$
\end{prop}

\begin{lemma}\label{nn-1} Let $\Delta u\in L^1(\Omega)$ and $u=0$ on $\partial\Omega$, where $\Omega\subset\R{n}$ is a bounded domain. Then for every $1\leq p<\frac{n}{n-1}$ we have
$$\|u\|_{W^{1,p}(\Omega)}\leq C(p)\|\Delta u\|_{L^1(\Omega)}$$
\end{lemma}

\begin{proof} Let $u\in C^\infty(\overline \Omega)$ and $u|_{\partial \Omega}=0$. If $1\leq p<\frac{n}{n-1}$, then $q:=\frac{p}{p-1}>n$. 
From $L^p$-theory (see e.g. \cite[Pag. 91]{Sim}) and the imbedding $W^{1,q}\hookrightarrow L^\infty$ we infer
\begin{eqnarray*}
\|\nabla u\|_{L^p(\Omega)}&\leq& C\sup_{\substack{\varphi\in W^{1,q}_0(\Omega)\\ \|\nabla\varphi\|_{L^q(\Omega)}\leq 1}}\int_{\Omega}\nabla u\cdot\nabla\varphi dx = C\sup_{\substack{\varphi\in W^{1,q}_0(\Omega)\\ \|\nabla\varphi\|_{L^q(\Omega)}\leq 1}}\int_\Omega -\Delta u\varphi dx\\
&\leq&C\sup_{\substack{\varphi\in L^\infty(\Omega)\\ \|\varphi\|_{L^\infty(\Omega)}\leq 1}}\int_\Omega -\Delta u\varphi dx\leq C\|\Delta u\|_{L^1}.
\end{eqnarray*}
To estimate $\|u\|_{L^p(\Omega)}$ we use Poincar\'e's inequality.
For the general case one can use a standard mollifying procedure.
\end{proof}

\medskip

\noindent\emph{Proof of Lemma \ref{propmu}.}
By Lemma \ref{nn-1}, $\|\Delta^{m-1}u\|_{W^{1,r}(\Omega)}\leq C(r) \|f\|_{L^1(\Omega)}$ for $1\leq r<\frac{2m}{2m-1}$. Then, by $L^p$-theory, $\|u\|_{W^{2m-1,r}(\Omega)}\leq C(r)\|f\|_{L^1(\Omega)}$, and by Sobolev's embedding,
\begin{equation}\label{Ls}
\|u\|_{L^s(\Omega)}\leq C(s)\|f\|_{L^1(\Omega)}, \quad \textrm{for all } 1\leq s<\infty.
\end{equation}
Now fix $B=B_{4R}(x_0)\subset\subset (\Omega\bs S_1)$ and write $u=u_1+u_2$, where
$$\left\{
\begin{array}{ll}
(-\Delta)^m u_2=f & \textrm{in }B_{4R}(x_0)\\
\Delta^j u_2=0 &\textrm{on }\partial B_{4R}(x_0) \textrm{ for } 0\leq j\leq m-1.
\end{array}
\right.$$
By $L^p$-theory
\begin{equation}\label{W2m}
\|u_2\|_{W^{2m,p}(B_{4R}(x_0))}\leq C(p,B)\|f\|_{L^p(B_{4R}(x_0))},
\end{equation}
with $C(p,B)$ depending on $p$ and the chosen ball $B$.
Together with \eqref{Ls}, we find
$$\|u_1\|_{L^1(B_{4R}(x_0))}\leq C(p,B) (\|f\|_{L^p(B_{4R}(x_0))}+\|f\|_{L^1(\Omega)}).$$ By Proposition \ref{c2m} 
$$\|u_1\|_{W^{2m,p}(B_R(x_0))}\leq  C(p,B) (\|f\|_{L^p(B_{4R}(x_0))}+\|f\|_{L^1(\Omega)}),$$
and \eqref{21*} follows. \hfill$\square$

\begin{prop}\label{propliou} Let $S=\{x_1,\ldots ,x_I\}\subset\R{2m}$ be a finite set and let $h\in C^\infty(\R{2m}\backslash S)$ satisfy $\Delta^m h=0$ and
\begin{equation}\label{nablah}
\dist(x,S)|\nabla h(x)|\leq C,\quad \text{for }x\in \R{2m}\backslash S.
\end{equation}
Then there are constants $\beta$ and  $\lambda_i$, $1\leq i\leq I$,  such that
\begin{equation}\label{h(x)}
h(x)=\sum_{i=1}^I \lambda_i \log\frac{1}{|x-x_i|}+\beta.
\end{equation}
\end{prop}

\begin{proof} Thanks to $\eqref{nablah}$, $h\in L^1_{\loc}(\R{2m})$, so that $\Delta^m h$ is well defined in the sense of distributions and it is supported in $S$. Therefore
$$\Delta^m h=\sum_{i=1}^I \beta_i\delta_{x_i},$$
for some constants $\beta_i$. Then, recalling \eqref{fund}, if we set
$$v(x):=h(x)-\sum_{i=1}^I \lambda_i\log \frac{1}{|x-x_i|}, \quad \lambda_i:=(-1)^m\frac{\beta_i}{\gamma_m},$$
we get $\Delta^m v\equiv 0$ in $\R{2m}$ in the sense of distributions (hence $v$ is smooth) and 
\begin{equation}\label{nablav}
|\nabla v(x)||x|\leq C\quad \text{in }\R{2m}.
\end{equation}
Then $|v(x)|\leq C(\log(1+|x|)+1)$. By Theorem \ref{trmliou} $v$ is a polynomial, which \eqref{nablav} forces to be constant, say $v\equiv -\beta$. Now \eqref{h(x)} follows at once.
\end{proof}


\begin{thebibliography}{2}
\bibitem[Ada]{ada} \textsc{D. Adams} \emph{A sharp inequality of J. Moser for higher order derivatives}, Ann. of Math. \textbf{128} (1988), 385-398.
\bibitem[ARS]{ARS} \textsc{Adimurthi, F. Robert, M. Struwe} \emph{Concentration phenomena for Liouville's equation in dimension 4}, J. Eur. Math. Soc. \textbf{8} (2006), 171-180.
\bibitem[ADN]{ADN} \textsc{S. Agmon, A. Douglis, L. Niremberg} \emph{Estimates near the boundary for solutions of elliptic partial differential equations satisfying general boundary conditions}, Comm. Pure Appl. Math. \textbf{12} (1959), 623-727.
\bibitem[Ale1]{ale} \textsc{S. Alexakis} \emph{On the decomposition of global conformal invariants II}, Adv. Math. \textbf{206} (2006), 466-502.
\bibitem[Ale2]{ale2} \textsc{S. Alexakis} \emph{The decomposition of Global Conformal Invariants: On a conjecture of Deser and Schwimmer}, preprint (2007).
\bibitem[Bra]{bra} \textsc{T. Branson} \emph{The functional determinant}, Global Analysis Research Center Lecture Notes Series, no. 4, Seoul National University (1993).
\bibitem[BO]{BO} \textsc{T. Branson, B. Oersted} \emph{Explicit functional determinants in four dimensions}, Comm. Partial Differential Equations \textbf{16} (1991), 1223-1253.
\bibitem[BM]{BM} \textsc{H. Br\'ezis, F. Merle} \emph{Uniform estimates and blow-up behaviour for solutions of $-\Delta u=V(x)e^u$ in two dimensions}, Comm. Partial Differential Equations \textbf{16} (1991), 1223-1253.
\bibitem[Cha]{cha} \textsc{S-Y. A. Chang} \emph{Non-linear Elliptic Equations in Conformal Geometry}, Zurich lecture notes in advanced mathematics, EMS (2004).
\bibitem[CC]{CC} \textsc{S-Y. A. Chang, W. Chen} \emph{A note on a class of higher order conformally covariant equations}, Discrete Contin. Dynam. Systems \textbf{63} (2001), 275-281.
\bibitem[Ch]{ch} \textsc{X. Chen} \emph{Remarks on the existence of branch bubbles on the blowup analysis on equation $\Delta u=e^{2u}$ in dimension $2$},  Comm. Anal. Geom. \textbf{7} (1999), 295-302.
\bibitem[Che]{che} \textsc{S-S. Chern} \emph{A simple intrinsic proof of the Gauss-Bonnet theorem for closed Riemannian manifolds}, Ann. of Math. \textbf{45} (1944), 747-752.
\bibitem[DAS]{DAS} \textsc{A. Dall'Acqua, G. Sweers} \emph{Estimates for Green function and Poisson kernels of higher-order Dirichlet boundary value problems}, J. Differential Equations \textbf{205} (2004), 466-487.
\bibitem[Dru1]{dru1} \textsc{O. Druet} \emph{From one bubble to several bubbles: the low dimensional case}, J. Diff. Geom. \textbf{63} (2003), 399-473.
\bibitem[Dru2]{dru2} \textsc{O. Druet} \emph{Compactness for the Yamabe equation in low dimensions}, I.M.R.N \textbf{23} (2004), 1143-1191.
\bibitem[DH]{DH} \textsc{O. Druet, E. Hebey} \emph{Blow-up examples for second order elliptic PDEs of critical Sobolev growth}, Trans. Amer. Math. Soc. \textbf{357} (2005), 1915-1929.
\bibitem[DHR]{DHR} \textsc{O. Druet, E. Hebey, F. Robert} \emph{Blow-up theory for elliptic PDEs in Riemannian Geometry}, Mathematical Notes \textbf{45}, Princeton University Press, (2004).
\bibitem[DR]{DR} \textsc{O. Druet, F. Robert} \emph{Bubbling phenomena for fourth-order four-dimensional PDEs with exponential growth}, Proc. Amer. Math. Soc. \textbf{3} (2006), 897-908.
\bibitem[GT]{GT} \textsc{D. Gilbarg, N. Trudinger} \emph{Elliptic partial differential equations of second order}, Springer (1977).
\bibitem[GJMS]{GJMS} \textsc{C. R. Graham, R. Jenne, L. Mason, G. Sparling} \emph{Conformally invariant powers of the Laplacian, I: existence}, J. London Math. Soc. \textbf{46} no.2 (1992), 557-565.
\bibitem[H\'el]{Hel} \textsc{F. H\'elein} \emph{Harmonic maps, conservation laws and moving frames,} second edition, Cambridge University press (2002).
\bibitem[LS]{LS} \textsc{Y. Li, I. Shafrir} \emph{Blow-up analysis for solutions of $-\Delta u = Ve^u$ in dimension $2$}, Indiana Univ. Math. J. \textbf{43} (1994), 1255-1270.
\bibitem[Lin]{lin} \textsc{C. S. Lin} \emph{A classification of solutions of conformally invariant fourth order equations in $\R{n}$}, Comm. Math. Helv. \textbf{73} (1998), 206-231.
\bibitem[Mal]{mal} \textsc{A. Malchiodi} \emph{Compactness of solutions to some geometric fourth-order equations}, J. reine angew. Math. \textbf{594} (2006), 137-174.
\bibitem[MS]{MS} \textsc{A. Malchiodi, M. Struwe} \emph{$Q$-curvature flow on $S^4$}, J. Diff. Geom. \textbf{73} (2006), 1-44.
\bibitem[Mar1]{mar1} \textsc{L. Martinazzi} \emph{Classification of the entire solutions to the higher order Liouville's equation on $\R{2m}$}, to appear in Math. Z.
\bibitem[Mar2]{mar2} \textsc{L. Martinazzi} \emph{Conformal metrics on $\R{2m}$ with non-positive constant $Q$-curvature}, Rend. Lincei Mat. Appl. \textbf{19} (2008), 279-292.
\bibitem[Mar3]{mar3} \textsc{L. Martinazzi} \emph{A threshold phenomenon for embeddings of $H^m_0$ into Orlicz spaces}, Preprint (2009).
\bibitem[Ndi]{ndi} \textsc{C. B. Ndiaye} \emph{Constant $Q$-curvature metrics in arbitrary dimension}, J. Funct. Anal. \textbf{251} (2007), 1-58.
\bibitem[O'N]{ON} \textsc{R. O'Neil} \emph{Convolution operators and $L(p,q)$ spaces}, Duke Math. J. \textbf{30} (1963), 129-142.
\bibitem[Pan]{pan} \textsc{S. Paneitz} \emph{A quartic conformally covariant differential operator for arbitrary pseudo-Riemannian manifolds}, SIGMA Symmetry Integrability Geom. Methods Appl. \textbf{4} (2008), Paper 036. Preprint (1983).
\bibitem[Piz]{Piz} \textsc{P. Pizzetti} \emph{Sulla media dei valori che una funzione dei punti dello spazio assume alla superficie di una sfera}, Rend. Lincei \textbf{18} (1909), 182-185.
\bibitem[Rob1]{rob1} \textsc{F. Robert} \emph{Concentration phenomena for a fourth order equation with exponential growth: the radial case}, J. Differential Equations \textbf{231} (2006), 135-164.
\bibitem[Rob2]{rob2} \textsc{F. Robert} \emph{Quantization effects for a fourth order equation of exponential growth in dimension four}, Proc. Roy. Soc. Edinburgh Sec. A \textbf{137} (2007), 531-553.
\bibitem[RS]{RS} \textsc{F. Robert, M. Struwe} \emph{Asymptotic profile for a fourth order PDE with critical exponential growth in dimension four}, Adv. Nonlin. Stud. \textbf{4} (2004), 397-415.
\bibitem[Sim]{Sim} \textsc{C. G. Simader} \emph{On Dirichlet's boundary value problem}, Lecture Notes in Mathematics n. 268, Springer 1972.
\bibitem[Str]{str} \textsc{M. Struwe} \emph{A flow approach to Nirenberg's problem}, Duke Math. J. \textbf{128}(1) (2005), 19-64.
\bibitem[Xu]{xu} \textsc{X. Xu} \emph{Uniqueness and non-existence theorems for conformally invariant equations}, J. Funct. Anal. \textbf{222} (2005), 1-28.
\end{thebibliography}
\end{document}